\documentclass[12pt]{article}

\usepackage{amsmath,amsthm,amssymb,amsfonts}
\usepackage{latexsym,enumerate}
\usepackage{mathrsfs}
\usepackage{epsfig,color,a4wide}
\usepackage{hyperref}

\def\N{\mathbb{N}}
\def\Z{\mathbb{Z}}

\def\R{\mathbb{R}}

\DeclareMathOperator{\argmin}{argmin}

\def\eps{\varepsilon}

\newtheorem{theorem}{Theorem}[section]
\newtheorem{lemma}[theorem]{Lemma}
\newtheorem{proposition}[theorem]{Proposition}

\theoremstyle{definition}
\newtheorem{definition}[theorem]{Definition}

\theoremstyle{remark}
\newtheorem{remark}[theorem]{Remark}
\newtheorem{example}[theorem]{Example}

\title{A junction condition by specified homogenization\\
and application to traffic lights}

 \author{G. Galise\footnote{Department of Mathematics, University of Salerno, Via Giovanni Paolo II, 132,
84084 Fisciano (SA), Italy}, \;
   C. Imbert\footnote{CNRS, UMR 7580, Universit\'e Paris-Est
     Cr\'eteil, 61 avenue du G\'en\'eral de Gaulle, 94 010 Cr\'eteil
     cedex, France}, \; R. Monneau\footnote{Universit\'e Paris-Est, CERMICS (ENPC), 
6-8 Avenue Blaise Pascal, Cit\'e Descartes, Champs-sur-Marne,
F-77455 Marne-la-Vall\'ee Cedex 2, France}}

\begin{document}

\maketitle

\begin{abstract}
  Given a coercive Hamiltonian which is quasi-convex with respect to
  the gradient variable and periodic with respect to time and space at
  least ``far away from the origin'', we consider the solution of the
  Cauchy problem of the corresponding Hamilton-Jacobi equation posed
  on the real line. Compact perturbations of coercive periodic
  quasi-convex Hamiltonians enter into this framework for example.  We
  prove that the rescaled solution converges towards the solution of
  the expected effective Hamilton-Jacobi equation, but whose ``flux''
  at the origin is ``limited'' in a sense made precise by the authors in
  \cite{im}. In other words, the homogenization of such a
  Hamilton-Jacobi equation yields to supplement the expected
  homogenized Hamilton-Jacobi equation with a junction condition at
  the single discontinuous point of the effective Hamiltonian. We also
  illustrate possible applications of such a result by deriving, for a
  traffic flow problem, the effective flux limiter generated by the
  presence of a finite number of traffic lights on an ideal road. We
  also provide meaningful qualitative properties of the effective
  limiter.
\end{abstract}


junction condition, flux-limited
solution, viscosity solution.

\section{Introduction}

\subsection{Setting of the general problem}

This article is concerned with the study of the limit of the solution
$u^\varepsilon(t,x)$ of the following equation
\begin{equation}\label{eq:hj-eps}
u^\eps_t +H\left(\frac{t}\eps,\frac{x}\eps,u^\eps_x\right) = 0 \quad \text{ for }  
(t,x)\in(0,T)\times \R
\end{equation}
submitted to the initial condition
\begin{equation}\label{eq:ic}
u^\eps(0,x)=u_0(x) \quad \text{for } x\in  \R
\end{equation}
for a Hamiltonian $H$ satisfying the following assumptions:
\begin{itemize}
\item[\textbf{(A0)}] (Continuity) $H\colon \R^3 \to \R$ is continuous.
\item[\textbf{(A1)}] (Time periodicity) For all $k\in \Z$ and
  $(t,x,p)\in \R^3$, 
\[H(t+k,x,p)=H(t,x,p).\]
\item[\textbf{(A2)}] (Uniform modulus of continuity in time)
There exists a modulus of continuity $\omega$ such that for all
$t,s,x,p \in \R$, 
\[H(t,x,p)-H(s,x,p)\le \omega(\left|t-s\right|\left(1+\max
\left(H(s,x,p),0\right)\right)).\]
\item[\textbf{(A3)}] (Uniform coercivity)
\[\lim_{|q|\to +\infty} H(t,x,q)=+\infty\]
uniformly with respect to $(t,x)$. 
\item[\textbf{(A4)}] (Quasi-convexity of $H$ for large $x$'s) There
 exists some $\rho_0>0$ such that for all $x\in \R\setminus
 (-\rho_0,\rho_0)$, there exists a continuous map $t\mapsto p^0(t,x)$
 such that
\[\left\{\begin{array}{l}
H(t,x,\cdot) \quad \text{is non-increasing in}\quad (-\infty,p^0(t,x)),\\
H(t,x,\cdot) \quad \text{is non-decreasing in}\quad (p^0(t,x),+\infty).
\end{array}\right.\]
\item[\textbf{(A5)}] (Left and right Hamiltonians)
There exist two Hamiltonians $H_\alpha(t,x,p)$, $\alpha=L,R$,
such that \[\left\{\begin{array}{l}
H(t,x+k,p) -H_L(t,x,p)\to 0 \quad \text{as}\quad \Z\ni k\to -\infty \\
H(t,x+k,p) -H_R(t,x,p)\to 0 \quad \text{as}\quad \Z\ni k\to +\infty
\end{array}\right.\]
uniformly with respect to $(t,x,p)\in [0,1]^2\times \R$, and 
for all $k,j\in \Z$, $(t,x,p)\in \R^3$ and $\alpha \in \{L,R\}$,  
\[H_\alpha(t+k,x+j,p)=H_\alpha(t,x,p).\]
\end{itemize}

We have to impose some condition in order to ensure that effective Hamiltonians $\bar H_\alpha$ are quasi-convex; indeed, we will see that the effective equation
should be solved with \emph{flux-limited solutions} recently introduced by 
the the second and third authors \cite{im}; such a theory relies on the 
quasi-convexity of the Hamiltonians. 
\begin{itemize}
\item[\textbf{(B-i)}] (Quasi-convexity of the left and right
  Hamiltonians) For each $\alpha=L,R$, $H_\alpha$ does not depend on
  time and there exists $p_\alpha^0$ (independent on $(t,x)$) such that
\[\left\{\begin{array}{l}
H_\alpha(x,\cdot) \quad \text{is non-increasing on}\quad (-\infty,p^0_\alpha),\\
H_\alpha(x,\cdot) \quad \text{is non-decreasing on}\quad (p^0_\alpha,+\infty).
\end{array}\right.\]
\item[\textbf{(B-ii)}] (Convexity of the left and right Hamiltonians)
 For each $\alpha=L,R$, and for all $(t,x)\in \R\times \R$, the map
 $p\mapsto H_\alpha(t,x,p)$ is convex.
\end{itemize}
\begin{example}
A simple example of such a Hamiltonian is
\[H(t,x,p)= \left|p\right| - f(t,x)\]
with a continuous function $f$ satisfying $f(t+1,x)=f(t,x)$ and
$f(t,x)\to 0$ as $\left|x\right|\to +\infty$ uniformly with respect to
$t\in \R$.
\end{example}

\subsection{Main results}

Our main result is concerned with the limit of the solution $u^\eps$
of \eqref{eq:hj-eps}-\eqref{eq:ic}. It makes part of the huge
literature dealing with homogenization of Hamilton-Jacobi equation,
starting with the pioneering work of Lions, Papanicolaou and Varadhan
\cite{lpv}. In particular, we need to use the perturbed test function
introduced by Evans \cite{evans}. As pointed out to us by the referee,
there are few papers dealing with Hamiltonians that depend on time; it
implies in particular that so-called correctors also depend on time.
The reader is in particular referred to \cite{bs,br} for the large
time behaviour and to \cite{fim09a,fim09b,fim12} for homogenization
results.  This limit satisfies an effective Hamilton-Jacobi equation
posed on the real line whose Hamiltonian is discontinuous. More
precisely, the effective Hamiltonian equals the one which is expected
(see \textbf{(A5)}) in $(-\infty;0)$ and $(0;+\infty)$; in particular,
it is discontinuous in the space variable (piecewise constant in
fact).  In order to get a unique solution, a flux limiter should be
identified \cite{im}.

\subsubsection*{Homogenized Hamiltonians and effective flux limiter}

The homogenized left and right Hamiltonians are classically determined
by the study of some ``cell problems''.
\begin{proposition}[Homogenized left and right
    Hamiltonians]\label{prop:quasi-conv} Assume  \textbf{\upshape (A0)-(A5)},
  and either \textbf{\upshape (B-i)} or \textbf{\upshape (B-ii)}. Then for every
  $p\in \R$, and $\alpha=L,R$, there exists a unique $\lambda\in\R$
  such that there exists a bounded solution $v^\alpha$ of
\begin{equation}\label{eq:cell-alpha}
\left\{\begin{array}{l}
v^\alpha_t + H_\alpha(t,x,p+v^\alpha_x)=\lambda \quad \text{in}\quad \R\times \R,\\
v^\alpha \text{ is $\Z^2$-periodic}.
\end{array}\right.
\end{equation}
If $\bar H_\alpha(p)$ denotes such a $\lambda$, then the map $p\mapsto
\bar H_\alpha(p)$ is continuous, coercive and quasi-convex.
\end{proposition}
\begin{remark}
  We recall that a function $\bar H_\alpha$ is quasi-convex if the
  sets $\{ \bar H_\alpha \le \lambda \}$ are convex for all $\lambda
  \in \R$. If $\bar H_\alpha$ is also coercive, then $\bar p^0_\alpha$
   denotes in proofs some $p \in \argmin \bar H_\alpha$. 
\end{remark}
The effective flux limiter $\bar A$ is the smallest $\lambda \in \R$ for which
there exists a solution $w$ of the following global-in-time
Hamilton-Jacobi equation
\begin{equation}\label{eq:cell}
\left\{\begin{array}{ll}
w_t + H(t,x,w_x)=\lambda, \quad (t,x)\in \R\times \R,\\
w \text{ is $1$-periodic w.r.t. } t.
\end{array}\right.
\end{equation}
\begin{theorem}[Effective flux limiter]\label{thm:bar-A}
Assume \textbf{\upshape (A0)-(A5)} and either \textbf{\upshape (B-i)}
  or \textbf{\upshape (B-ii)}.  The set 
\[ E= \{ \lambda \in \R: \exists \text{$w$ sub-solution of
  \eqref{eq:cell}} \} \]
is not empty and  bounded from below. Moreover, if $\bar A$ denotes
the infimum of $E$, then
\begin{equation}\label{eq::g1}
\bar A\ge A_0:=\max_{\alpha=L,R}\left(\min \bar H_\alpha\right).
\end{equation}
\end{theorem}
\begin{remark}
We will see below (Theorem~\ref{thm:corrector}) that the infimum is in
fact a minimum: there exists a global corrector which, in
particular, can be rescaled properly.
\end{remark}

We can now define the effective junction condition. 
\begin{definition}[Effective junction condition]\label{defi:F-bar}
The \emph{effective junction function} $F_{\bar A}$ is defined by
\[F_{\bar A}(p_L,p_R):=\max
(\bar A,\bar H_L^+(p_L), \bar H_R^-(p_R))\]
where
\[\bar H_\alpha^-(p)=\left\{\begin{array}{ll}
\bar H_\alpha (p) &\quad \text{if}\quad p< \bar p_\alpha^0,\\
\bar H_\alpha(\bar p_\alpha^0)&\quad \text{if}\quad p\ge  \bar p_\alpha^0
\end{array}\right.
\quad \text{and}\quad 
\bar H_\alpha^+(p)=\left\{\begin{array}{ll}
\bar H_\alpha(\bar p_\alpha^0) &\quad \text{if}\quad p\le \bar p_\alpha^0,\\
\bar H_\alpha (p) &\quad \text{if}\quad p>  \bar p_\alpha^0
\end{array}\right.\]
where $\bar p_\alpha^0 \in \argmin \bar H_\alpha$. 
\end{definition}

\subsubsection*{The convergence result}

Our main result is the following theorem.
\begin{theorem}[Junction condition by homogenization]\label{thm:conv}
Assume \textbf{\upshape (A0)-(A5)} and either \textbf{\upshape (B-i)}
or \textbf{\upshape (B-ii)}. Assume that the initial datum $u_0$ is
Lipschitz continuous and for $\eps>0$, let $u^\eps$ be the solution of
\eqref{eq:hj-eps}-\eqref{eq:ic}.  Then $u^\eps$ converges locally
uniformly to the unique flux-limited solution $u^0$ of
\begin{equation}\label{eq:hj-homog}
\left\{\begin{array}{ll}
u^0_t + \bar H_L(u^0_x)=0, & t>0,x<0,\\
u^0_t + \bar H_R(u^0_x)=0, & t>0,x>0,\\
u^0_t + F_{\bar A}(u^0_x(t,0^-),u^0_x(t,0^+))=0, & t>0,x=0
\end{array}\right.
\end{equation}
submitted to the initial condition \eqref{eq:ic}. 
\end{theorem}
\begin{remark}
We recall that the notion of flux-limited solution for
  (\ref{eq:hj-homog}) is introduced in \cite{im}.
\end{remark}
This theorem asserts in particular that the slopes of
  the limit solution at the origin are characterized by the effective
  flux limiter $\bar A$. Its proof relies on the construction of a
  global ``corrector'', \textit{i.e.} a solution of \eqref{eq:cell},
  which is close to an appropriate $V$-shaped function after
  rescaling. This latter condition is necessary so that the slopes at
  infinity of the corrector fit the expected slopes of the solution of
  the limit problem at the origin. Here is a precise statement.
\begin{theorem}[Existence of a global corrector for the junction]\label{thm:corrector-simple}
Assume \textbf{\upshape (A0)-(A5)} and either \textbf{\upshape (B-i)}
or \textbf{\upshape (B-ii)}.
There exists a solution $w$ of
  \eqref{eq:cell} with $\lambda = \bar A$ such that, the function
\[w^\eps(t,x)=\eps w(\eps^{-1}t,\eps^{-1}x)\]
 converges locally uniformly (along a subsequence $\eps_n\to
0$) towards a function $W=W(x)$ which satisfies $W(0)=0$ and
\begin{equation}\label{eq:W-estim-2}
\hat p_R x 1_{\left\{x>0\right\}} + \hat p_L x
1_{\left\{x<0\right\}}\ge W(x)\ge \bar p_R x
1_{\left\{x>0\right\}} + \bar p_L x 1_{\left\{x<0\right\}}
\end{equation}
where
\begin{align}
\label{def:pR}
\begin{cases} \bar p_R = \min E_R \\ \hat p_R = \max E_R \end{cases}
& \quad \text{with}\quad E_R:= \left\{p\in \R,\quad  \bar H_R^+(p)=\bar
H_R(p)=\bar A\right\} \\
\label{def:pL}
\begin{cases}\bar p_L = \max E_L \\\hat p_L = \min E_L \end{cases}
& \quad  \text{with}\quad E_L:= \left\{p\in \R,\quad  \bar
H_L^-(p)=\bar H_L(p)=\bar A\right\}.
\end{align}
\end{theorem}
The construction of this global corrector is the reason why
homogenization is referred to as being ``specified''. See also
Section~\ref{subsec:rel} about related results.  As a matter of fact,
we will prove a stronger result, see Theorem~\ref{thm:corrector}.

\subsubsection*{Extension: application to traffic lights}

The techniques developed to prove the Theorem~\ref{thm:conv} allow us
to deal with a different situation inspired from traffic flow
problems. As explained in \cite{imz}, such problems are related to
the study of some Hamilton-Jacobi equations. The problem that we address in
Theorem~\ref{thm:conv-time} below is motivated by its meaningful
application to traffic lights. We aim at figuring out how the fraffic
flow on an ideal (infinite, straight) road is modified by the presence
of a finite number of traffic lights.

We can consider a Hamilton-Jacobi equation whose Hamiltonian does not
depend on $(t,x)$ for $x$ outside a (small) interval of the form
$N_\eps = (b_1 \eps, b_N \eps)$ and is piecewise constant with respect
to $x$ in $(b_1 \eps,b_N \eps)$.  At space discontinuities, junction
conditions are imposed with $\eps$-time periodic flux limiters.  The
limit solution satifies the equation after the ``neighbourhood''
$N_\eps$ disappeared. We will see that the equation keeps memory of
what happened there through a flux limiter at the origin
$x=0$. \medskip

Let us be more precise now.  For $N\ge 1$, (a finite number of)
junction points $-\infty=b_0<b_1<b_2<\dots <b_N<b_{N+1}=+\infty$ and
(a finite number of) times $0=\tau_0<\tau_1<\dots
<\tau_K<1=\tau_{K+1}$, $K \in \N$ are given. For $N\ge 1$ and
$\alpha \in \{0,\dots,N\}$, $\ell_\alpha$ denotes
$b_{\alpha+1}-b_\alpha$. Note that $\ell_\alpha=+\infty$ for
$\alpha=0,N$.

We then consider the solution $u^\eps$ of \eqref{eq:hj-eps} where the
Hamiltonian $H$ satifies the following conditions.
\begin{itemize}
\item[\textbf{(C1)}] The Hamiltonian is given by
\[ H (t,x,p) = 
\begin{cases}\bar H_{\alpha}(p) & \text{ if }  b_{\alpha}
<x < b_{\alpha+1} \\
\max (\bar H_{\alpha-1}^+(p^-),\bar
H_{\alpha}^-(p^+), a_\alpha(t) ) & \text{ if } x=b_\alpha, \alpha
\neq 0.
\end{cases} \]
\item[\textbf{(C2)}] The Hamiltonians
  $\bar{H}_\alpha$, for $\alpha=0,\dots, N$, are continuous, coercive and quasi-convex.
\item[\textbf{(C3)}] The flux limiters 
$a_\alpha$, for  $\alpha=1,\dots,N$ and $i=0,\dots,K$, satisfy 
\[a_\alpha(s+1)=a_\alpha(s) \quad \text{with}\quad a_\alpha(s)=
A_\alpha^i \quad \text{for all}\quad s\in
\left[\tau_i,\tau_{i+1}\right)\]
with
$(A_\alpha^i)_{\alpha=1,\dots,N}^{i=0,\dots,K}$ satisfying
\(A_\alpha^i \ge \max_{\beta=\alpha-1,\alpha}\left(\min \bar
  H_\beta\right).\)
\end{itemize}
\begin{remark}
  The Hamiltonians outside $N_\eps$ are denoted by $\bar H_\alpha$
  instead of $H_\alpha$ in order to emphasize that they do not depend
  on time and space.
\end{remark}
\begin{remark}
In view of the litterature in traffic modeling, the Hamiltonians could
be assumed to be convex. But we prefer to stick to the quasi-convex framework
since it seems to us that it is the natural one (in view of \cite{im}). 
\end{remark}

The equation is supplemented with the following initial condition
\begin{equation}\label{eq:ic-bis}
u^\eps (0,x) = U^\eps_0 (x) \quad \text{ for } x \in \R
\end{equation}
with
\begin{equation}\label{hyp:uo}
 U^\eps_0 \text{ is equi-Lipschitz continuous and } U^\eps_0 \to u_0
\text{ locally uniformly}.
\end{equation}
Then the following convergence result holds true. 
\begin{theorem}[Time homogenization of traffic lights]\label{thm:conv-time}
  Assume \textbf{\upshape (C1)-(C3)} and \eqref{hyp:uo}. Let $u^\eps$ be the solution of
  \eqref{eq:hj-eps}-\eqref{eq:ic-bis} for all $\eps >0$. Then:
\begin{enumerate}[\upshape i)]
\item {\upshape (Homogenization)}
There exists some $\bar A\in \R$ such that $u^\eps$ converges
locally uniformly as $\eps$ tends to zero towards the unique viscosity
solution $u^0$ of \eqref{eq:hj-homog}-\eqref{eq:ic} with
\[\bar H_L:=\bar H_0,\quad \bar H_R:= \bar H_N.\]
\item {\upshape (Qualitative properties of $\bar A$)}
For $\alpha=1,\dots,N$, $\langle a_\alpha \rangle$ denotes $\int_0^1  a_\alpha(s)\ ds$.
The effective limiter $\bar A$ satisfies the following properties.
\begin{itemize}
\item For all $\alpha$,  $\bar A$ is non-increasing w.r.t. $\ell_\alpha$. 
\item For $N=1$, 
\begin{equation}\label{eq:N1}
\bar A = \langle a_1 \rangle. 
\end{equation}
\item For  $N \ge 1$,
\begin{equation}\label{eq:Nge1}
\bar A \ge  \max_{\alpha=1,\dots,N}\quad \langle a_\alpha \rangle .
\end{equation}
\item For $N \ge 2$, there exists a critical distance $d_0 \ge 0$  such that
\begin{equation}\label{eq:critical}
\bar A =  \max_{\alpha=1,\dots,N}\quad \langle a_\alpha \rangle  \quad
\quad \mbox{if}\quad \quad \min_\alpha \ell_\alpha \ge d_0;
\end{equation}
this distance $d_0$ only depends on $\displaystyle
\max_{\alpha=1,\dots,N} \|a_\alpha\|_\infty$, $\displaystyle
\max_{\alpha=1,\dots,N} \langle a_\alpha\rangle$ and the $\bar
H_\alpha$'s.
\item We have
\begin{equation}\label{eq:limit}
\bar A \to \langle \bar a \rangle \quad  \mbox{as}\quad (\ell_1,\dots,\ell_{N-1}) \to (0,\dots,0)
\end{equation}
where $\bar a(\tau)=\max_{\alpha=1,\dots,N} a_\alpha(\tau)$.
\end{itemize}
\end{enumerate}
\end{theorem}
\begin{remark} \label{rem:meaning} Since the function $a(t)$ is
  piecewise constant, the way $u^\eps$ satisfies \eqref{eq:hj-eps} has
  to be made precise. An $L^1$ theory in time (following for instance
  the approach of \cite{b1,b2}) could probably be developed for such a
  problem, but we will use here a different, elementary approach.  The
  Cauchy problem is understood as the solution of successive Cauchy
  problems. This is the reason why we will first prove a global
  Lipschitz bound on the solution so that there indeed exists such a
  solution.
\end{remark}
\begin{remark}\label{rem::g1}
  Note that the result of Theorem \ref{thm:bar-A} still holds for
  equation \eqref{eq:hj-eps} under Assumptions \textbf{\upshape
    (C1)-(C3)}, with the set $E$ defined for sub-solutions which are
  moreover assumed to be globally Lipschitz (without fixed bound on
  the Lipschitz constant). The reader can check that the proof is
  unchanged.
\end{remark}
\begin{remark}
  It is somewhat easy to get \eqref{eq:N1} when the Hamiltonians $\bar
  H_\alpha$ are convex by using the optimal control interpretation of
  the problem.  In the more general case of quasi-convex Hamiltonians,
  the result still holds true but the proof is more involved.
\end{remark}
\begin{remark}\label{rem::g81}
  We may have $\bar A > \max_{\alpha=1,\dots,N}\langle a_\alpha
  \rangle$.  It is possible to deduce it from \eqref{eq:limit} in the
  case $N=2$ by using the traffic light interpretation of the problem.
  If we have two traffic lights very close to each other (let us say
  that the distance in between is at most the place for only one car),
  and if the common period of the traffic lights are exactly in
  opposite phases (with for instance one minute for the green phase,
  and one minute for the red phase), then the effect of the two
  traffic lights together, gives a very low flux which is much lower
  than the effect of a single traffic light alone (\textit{i.e.} here
  at most one car every two minutes will go through the two traffic
  lights).
\end{remark}

\subsection{Traffic flow interpretation of Theorem~\ref{thm:conv-time}}

  We mentioned above that there are some connections between our
  problem and traffic flows.  

Inequality \eqref{eq:Nge1} has a natural
traffic interpretation, saying that the average limitation on the
traffic flow created by several traffic lights on a single road is
higher or equal to the one created by the traffic light which creates
the highest limitation.  Moreover this average limitation is smaller
if the distances between traffic lights are bigger, as says the
monotonicity of $\bar A$ with respect to the distances $\ell_\alpha$.

Property (\ref{eq:critical}) says that the minimal limitation is reached
if the distances between the traffic lights are bigger than a critical
distance $d_0$. The proof of this result is quite involved and is
reflected in the fact that the bounds that we have on $d_0$ are not
continuous on the data ($\displaystyle \max_{\alpha=1,\dots,N}
\|a_\alpha\|_\infty$, $\displaystyle \max_{\alpha=1,\dots,N}\langle
a_\alpha \rangle$ and the $\bar H_\alpha$'s).  

Finally property (\ref{eq:limit}) is very natural from the point of
view of traffic, since it corresponds to the case where all the
traffic lights would be at the same position.

\subsection{Related results}
\label{subsec:rel}

Achdou and Tchou \cite{at} studied a singular perturbation problem
which has the same flavor as the one we are looking at in the present
paper. More precisely, they consider the simplest network (a so-called
junction) embedded in a star-shaped domain. They prove that the value
function of an infinite horizon control problem converges, as the
star-shaped domain ``shrinks'' to the junction, to the value function
of a control problem posed on the junction. We borrow from them the
idea of studying the cell problem on truncated domains with state
constraints. We provide a different approach, which is also in some
sense more general because it can be applied to problems outside the
framework of optimal control theory. Our approach relies in an
essential way on the general theory developed in \cite{im}.

The general theme of Lions's 2013-2014 lectures at Coll\`ege de France
\cite{lions} is ``Elliptic or parabolic equations and specified
homogenization''. As far as first order Hamilton-Jacobi equations are
concerned, the term ``specified homogenization'' refers to the problem
of constructing correctors to cell problems associated with
Hamiltonians that are typically the sum of a periodic one $H$ and a
compactly supported function $f$ depending only on $x$, say. Lions
exhibits sufficient conditions on $f$ such that the effective
Hamilton-Jacobi equation is not perturbed. In terms of flux limiters
\cite{im}, it corresponds to look for sufficient conditions such that
the effective flux limiter $\bar A$ given by Theorem~\ref{thm:bar-A} is (less than or) equal to $A_0=
\min H$.

Barles, Briani and Chasseigne \cite[Theorem 6.1]{bbc} considered the
case 
\[ H(x,p ) = \varphi \left(\frac{x}\eps\right) H_R (p) +
 \left(1-\varphi\left(\frac{x}\eps\right)\right) H_L (p)\]
for some continuous increasing function $\varphi: \R \to \R$ such that
\[ \lim_{s\to -\infty} \varphi (s)=0 \quad \text{ and } \quad \lim_{s\to
  +\infty} \varphi (s) =1.\] They prove that $u^\eps$ converges
towards a value function denoted by $U^-$, that they characterize as
the solution to a particular optimal control problem. It is proved in
\cite{im} that $U^-$ is the solution of \eqref{eq:hj-homog} with $\bar
H_\alpha = H_\alpha$ and $\bar A$ replaced with $A_I^+= \max(A_0,A^*)$
with
\[A_0 = \max (\min H_R,\min H_L) \quad \text{ and } \quad A^*  =
\max_{q \in [\min(p_R^0,p_L^0),\max(p_R^0,p_L^0)]} (\min (H_R(q),H_L(q))). \]

In \cite{gh}, Giga and Hamamuki develop a theory which allows in
particular to prove existence and uniqueness for the following
Hamilton-Jacobi equation (changing $u$ in $-u$) in $\R^d$,
\[\begin{cases} \partial_t u + |\nabla u | = 0 & \text{ for } x \neq 0 \\
 \partial_t u + |\nabla u | + c = 0 & \text{ at } x =0.\end{cases}\]
The solutions of \cite{gh} are constructed as limits of the following equation
\[ \partial_t u^\eps + |\nabla u^\eps | + c (1-|x|/\eps)^+ = 0. \]
In the monodimensional case ($d=1$), Theorem~\ref{thm:conv} implies that
$u^\eps$ converges towards 
\[\begin{cases} \partial_t u + |\nabla u | = 0 & \text{ for } x \neq 0 \\
 \partial_t u + \max(A,|\nabla u |) = 0 & \text{ at } x =0
\end{cases}\] for some $A \in \R$.  In view of
Theorem~\ref{thm:bar-A}, it is not difficult to prove that $A= \max
(0,c)$. The Hamiltonian $\max(c,|\nabla u |)$ is identified in
\cite{gh} and is referred to as the \emph{relaxed} one.

It is known that homogenization of Hamilton-Jacobi equations is
closely related to the study of the large time behaviour of
solutions. In \cite{hamamuki}, the large time behaviour
of Hamilton-Jacobi equations with discontinuous source terms is
discussed in two cases: for compactly supported ones and periodic
ones. Remark that in our setting, we can adress both and even the sum
of a periodic source term and of a compactly supported one. It would
be interesting to adress such a problem in the case of traffic lights. 
In \cite{jy}, the authors study the large time behaviour of the solutions 
of a Hamilton-Jacobi equations with an $x$-periodic Hamiltonian and 
what can be interpreted as a flux-limiter depending periodically in time.

\subsection{Further extensions}\label{ss1}

It is also possible to adress the time homogenization problem of
Theorem~\ref{thm:conv-time} with any finite number of junctions (with
limiter functions $a_\alpha(t)$ piecewise constants -- or continuous
-- and $1$-periodic), either separated with distance of order $O(1)$
or with distance of order $O(\varepsilon)$, or mixing both, and even
on a complicated network.  See also \cite{jy} for other connexions
between Hamilton-Jacobi equations and traffic light problems and
\cite{ags} for green waves modelling.

Note that the method presented in this paper can be readily applied
(without modifying proofs) to the study of homogeneization on a finite
number of branches and not only two branches; the theory developed in
\cite{im} should also be used for the limit problem.  

Similar questions in higher dimensions with point defects of other
co-dimensions will be addressed in future works.

\subsection{Organization of the article}

Section~\ref{sec:conv} is devoted to the proof of the convergence
result (Theorem~\ref{thm:conv}). Section~\ref{sec:homog} is devoted to
the construction of correctors far from the junction point
(Proposition~\ref{prop:quasi-conv}) while the junction case,
\textit{i.e.} the proof of Theorem~\ref{thm:corrector}, is addressed
in Section~\ref{sec:trunc}. We recall that
Theorem~\ref{thm:corrector-simple} is a straightforward corollary of
this stronger result. The proof of Theorem~\ref{thm:corrector} makes
use of a comparison principle which is expected but not completely
standard. This is the reason why a proof is sketched in Appendix,
together with two other ones that are rather standard but included for
the reader's convenience.

\paragraph{Notation.} A ball centered at $x$ of radius $r$ is denoted
by $B_r(x)$. If $\{u^\eps\}_\eps$ is locally bounded, the upper and
lower relaxed limits are defined as 
\[\begin{cases} 
\displaystyle \limsup_\eps{}^* u^\eps (X) = \limsup_{Y \to X, \eps \to 0} u^\eps
(Y), \\
\displaystyle \liminf_\eps{}_* u^\eps (X) = \liminf_{Y \to X, \eps \to 0} u^\eps
(Y).
\end{cases} \]
In our proofs, constants may change from line to line.


\section{Proof of convergence}
\label{sec:conv}

This section is devoted to the proof of Theorem~\ref{thm:conv}. We
first construct barriers. 
\begin{lemma}[Barriers]\label{Barriers}
There exists a nonnegative constant $C$ such that for any $\varepsilon>0$
\begin{equation}\label{barriers1}
\left|u^\varepsilon(t,x)-u_0(x)\right|\leq Ct\quad\text{for}\quad(t,x)\in(0,T)\times\R\,.
\end{equation}
\end{lemma}
\begin{proof}
Let $L_0$ be the Lipschitz constant of the initial datum $u_0$. Taking
\[C=\sup_{\stackrel{(t,x)\in\R\times\R}{|p|\leq
    L_0}}\left|H(t,x,p)\right|<+\infty\] owing to \textbf{(A0)} and
\textbf{(A5)}, the functions $u^\pm(t,x)=u_0(x)\pm Ct$ are a super-
and a sub-solution of (\ref{eq:hj-eps})-(\ref{eq:ic}) respectively and
(\ref{barriers1}) follows via comparison principle.
\end{proof}
We can now prove the convergence theorem. 
\begin{proof}[Proof of Theorem~\ref{thm:conv}]

We classically consider the upper and lower relaxed semi-limits
\[\begin{cases}
\overline u =\displaystyle \limsup_\eps{}^* u^\varepsilon,\\
\underline u =\displaystyle  \liminf_\eps{}_* u^\varepsilon\,.
\end{cases}\] Notice that these functions are well defined because of
Lemma \ref{Barriers}. In order to prove convergence of $u^\varepsilon$
towards $u^0$, it is sufficient to prove that $\overline u$ and
$\underline u$ are a sub- and a super-solution of
(\ref{eq:hj-homog})-(\ref{eq:ic}) respectively. The initial condition
immediately follows from (\ref{barriers1}). We focus our attention on
the sub-solution case since the super-solution one can be handled
similarly.

We first check that 
\begin{equation}\label{eq:wk}
 \overline u (t,0) = \limsup_{(s,y) \to (t,0), y >0} \overline u(s,y) = 
\limsup_{(s,y) \to (t,0), y <0} \overline u(s,y).
\end{equation}
This is a consequence of the stability of such a ``weak continuity''
condition, see \cite{im}. Indeed, it is shown in \cite{im} that 
classical viscosity solution can be viewed as flux-limited one; in
particular, $u^\eps$ solves 
\[ u^\eps_t + H^- \left(\frac{t}\eps,\frac{0}\eps,u^\eps_x (t,0^+) \right) \vee
H^+ \left(\frac{t}\eps,\frac{0}\eps,u^\eps_x (t,0^-)) \right) = 0 \quad \text{ for } t >0.\] 
Since these $\eps$-Hamiltonians are uniformly coercive and $u^\eps$ is continuous,
we conclude that \eqref{eq:wk} holds true. 

Let $\varphi$ be a test function such that 
\begin{equation}\label{eq::g2}
(\overline u-\varphi)(t,x)<(\overline u-\varphi)(\overline
t,\overline x)=0 \quad\forall(t,x)\in B_{\overline r}(\overline
t,\overline x)\setminus\left\{(\overline t,\overline x)\right\}.
\end{equation}
We argue by contradiction by assuming that
\begin{equation}\label{conv2}
  \varphi_t(\overline t,\overline x)+\bar H\left(\bar
x,\varphi_x(\overline t,\overline x)\right)=\theta>0,
\end{equation}
where 
\[\bar H\left(\bar x,\varphi_x(\overline t,\overline x)\right) := \left\{
\begin{array}{ll}
\bar H_R(\varphi_x(\overline t,\overline x)) &\quad \mbox{if}\quad \overline x >0,\\
\bar H_L(\varphi_x(\overline t,\overline x)) &\quad \mbox{if}\quad \overline x <0,\\
F_{\bar A}(\varphi_x(\overline t,0^-),\varphi_x(\overline t,0^+)) &\quad \mbox{if}\quad \overline x =0.
\end{array}\right.\]
We only treat the case where $\overline x=0$ since the case $\overline
x \neq 0$ is somewhat classical. This latter case is detailed in
Appendix for the reader's convenience. Using
\cite[Proposition~2.8]{im}, we may suppose that
\begin{equation}\label{eq::g3}
\varphi(t,x)=\phi(t)+\bar {p}_L x1_{\left\{x<0\right\}}+\bar {p}_R x 1_{\left\{x>0\right\}}
\end{equation}
where $\phi$ is a $C^1$ function defined in $(0,+\infty)$. In this case,
Eq.~\eqref{conv2} becomes
\begin{equation}\label{conv7}
\phi'(\bar t)+F_{\bar A}\left(\bar{p}_L,\bar{p}_R\right)=\phi'(\bar t)+\bar A=\theta>0.
\end{equation}
Let us consider a solution $w$ of the equation
\begin{equation}\label{conv8}
w_t+H(t,x,w_x)=\bar A
\end{equation}
provided by Theorem \ref{thm:corrector-simple}, which is in particular
$1$-periodic with respect to time. We recall that the function $W$ is
the limit of $w^\eps = \eps w (\cdot/\eps)$ as $\eps \to 0$. We claim that,
if $\varepsilon>0$ is small enough, the perturbed test function
$\varphi^\varepsilon(t,x)=\phi(t)+w^\varepsilon(t,x)$ \cite{evans} is a
viscosity super-solution of
\[\varphi^\varepsilon_t+H\left(\frac t\varepsilon, \frac
x\varepsilon,\varphi^\varepsilon_x\right)
=\frac\theta2\quad\text{in}\quad B_r(\overline t,0)\] for some
sufficiently small $r>0$. In order to justify this fact, let $\psi(t,x)$ be a test function
touching  $\varphi^\varepsilon$ from below at $(t_1,x_1)\in
B_r(\overline t,0)$. In this way
\[w\left(\frac{t_1}{\varepsilon},\frac{x_1}{\varepsilon}\right)
=\frac1\varepsilon\left(\psi(t_1,x_1)-\phi(t_1)\right)\]
and
\[w\left(s,y\right) \ge \frac1\varepsilon\left(\psi(\varepsilon
s,\varepsilon y) -\phi(\varepsilon s)\right)\]
for $(s,y)$ in a neighborhood of
$\left(\frac{t_1}{\varepsilon},\frac{x_1}{\varepsilon}\right)$. 
Hence from \eqref{conv7}-\eqref{conv8}
\begin{equation*}
\begin{split}
\psi_t(t_1,x_1)+H\left(\frac {t_1}\varepsilon, \frac
    {x_1}\varepsilon,\psi_x(t_1,x_1)\right)& \ge \bar A+\phi'(t_1)\\
&\geq\bar A+\phi'(\overline t)-\frac\theta2\geq\frac\theta2
\end{split}
\end{equation*}
provided $r$ is small enough. Hence, the claim is proved.

Combining (\ref{eq:W-estim-2}) from Theorem~\ref{thm:corrector-simple} 
with (\ref{eq::g2}) and (\ref{eq::g3}), we can 
fix $\kappa_r>0$ and $\eps>0$ small enough so that   
\[u^\varepsilon+\kappa_r \le \varphi^\eps \quad\text{on}\quad\partial B_r(\overline t,0).\]
By comparison principle the previous inequality holds in
$B_r(\overline t,0)$. Passing to the limit as $\eps \to 0$ and $(t,x)
\to (\bar t,\bar x)$,  we get the following contradiction
\[\overline u(\overline t,0)+\kappa_r\le\varphi(\overline
t,0)=\overline u(\overline t,0).\]
The proof of convergence is now complete.
\end{proof}
\begin{remark}\label{rem::1bis}
For the super-solution property, $\varphi$ in (\ref{eq::g3}) should be
replaced with 
\[\varphi(t,x)=\phi(t)+ \hat{p}_L x 1_{\left\{x<0\right\}} + \hat{p}_R x 1_{\left\{x>0\right\}}.\]
\end{remark}

\section{Homogenized Hamiltonians}
\label{sec:homog}

In order to prove Proposition~\ref{prop:quasi-conv}, we first prove
the following lemma. Even if the proof is standard, we give it in full
details since we will adapt it when constructing global correctors for
the junction. 
\begin{lemma}[Existence of a corrector]\label{lem:exis-base}
There exists $\lambda \in \R$ and a  bounded (discontinuous) viscosity solution of
\eqref{eq:cell-alpha}. 
\end{lemma}
\begin{remark}\label{rem:time-indep}
If $H_\alpha$ does not depend on $t$, then it is possible to construct
a corrector which does not depend on time either. We leave details to
the reader. 
\end{remark}
\begin{proof}
For any $\delta >0$, it is possible to construct a (possibly discontinuous)
viscosity solution $v^\delta$ of 
\[\begin{cases}
\delta v^\delta +v^\delta_t + H_\alpha (t,x,p+v_x^\delta) = 0 \quad
\text{ in } \R \times \R, \\
v^\delta \text{ is $\Z^2$-periodic}.
\end{cases}\]
First, the comparison principle implies 
\begin{equation}\label{eq:naive} 
|\delta v^\delta | \le C_\alpha
\end{equation}
where 
\[ C_\alpha = \sup_{(t,x)\in \left[0,1\right]^2} |H_\alpha (t,x,p)|.\]
Second, the function 
\[ m^\delta (x) = \sup_{t \in \R} (v^\delta)^* (t,x) \]
is a sub-solution  of 
\[ H_\alpha (t(x),x,p+m^\delta_x ) \le C_\alpha\] 
(for some function $t(x)$).  Assumptions \textbf{(A3)} and
\textbf{(A5)} imply in particular that there exists $C>0$ independent
of $\delta$ such that
\[ |m^\delta_x | \le C\] 
and
\[ v^\delta_t \le C.\]
In particular, the comparison principle implies that for all $t \in
\R$ and $x \in \R$ and $h \ge 0$, 
\[ v^\delta (t+h,x) \le v^\delta (t,x) + Ch.\]
Combining this inequality with the time-periodicity of $v^\delta$
yields 
\[ |v^\delta (t,x) - m^\delta (x) | \le C; \]
in particular,
\begin{equation}\label{eq::g4}
|v^\delta (t,x) - v^\delta (0,0)| \le C.
\end{equation}
Hence, the half relaxed limits 
\[ \bar v = \limsup_{\delta\to 0}{}^*(v^\delta - v^\delta (0,0)) \quad \text{
  and } \quad \underline v = \liminf_{\delta\to 0}{}_*(v^\delta - v^\delta
(0,0)) \]
are finite. Moreover,  \eqref{eq:naive} implies
that $\delta v^\delta (0,0) \to - \lambda$ (at least along a
subsequence). Hence, discontinuous stability of viscosity solutions
implies that $\bar v$ is a $\Z^2$-periodic sub-solution of \eqref{eq:cell-alpha} and
$\underline v$ is a  $\Z^2$-periodic super-solution of the same
equation. Perron's method then allows us to construct a corrector
between $\bar v$ and $\underline v + C$ with $C = \sup (\bar v -
\underline v)$. The proof of the lemma is now complete. 
\end{proof}
The following lemma is completely standard; the proof is given in
Appendix for the reader's convenience. 
\begin{lemma}[Uniqueness of $\lambda$]\label{lem:unique}
The real number $\lambda$ given by Lemma~\ref{lem:exis-base} is
unique. If $\bar H_\alpha(p)$ denotes such a real number, the function
$\bar H_\alpha$ is continuous. 
\end{lemma}

\begin{lemma}[Coercivity of $\bar H_\alpha$]\label{lem:coercive}
The continuous function $\bar H_\alpha$ is coercive,
\[ \lim_{|p|\to +\infty} \bar H_\alpha (p)= +\infty.\]
\end{lemma}
\begin{proof}
In view of the uniform coercivity in $p$ of $H_\alpha$ with respect to
$(t,x)$ (see \textbf{(A3)}), for any $R>0$ there exists a positive constant $C_R$ such
that
\begin{equation}\label{coercivity1}
|p|\geq C_R\quad\Rightarrow\quad \forall (t,x)\in\R\times\R, \quad H_\alpha(t,x,p)\geq R.
\end{equation}
Let $v^\alpha$ be the discontinuous corrector given by Lemma~\ref{lem:exis-base} and
$(\bar t,\bar x)$ be point of supremum of its upper semi-continuous
envelope $(v^\alpha)^*$. Then we have
$$H_\alpha(\bar t,\bar x,p)\le \bar H_\alpha(p)$$
which implies
\begin{equation}\label{coercivity2}
\bar H_\alpha(p)\geq R\quad\text{for}\quad |p|\geq C_R.
\end{equation}
The proof of the lemma is now complete. 
\end{proof}
We first prove the quasi-convexity of $\bar H_\alpha$ under
assumption~\textbf{(B-ii)}. We prove in fact more: the effective
Hamiltonian is convex in this case. 
\begin{lemma}[Convexity of $\bar H_\alpha$ under
    \textbf{(B-ii)}]\label{lem:conv}
    Assume \textbf{\upshape (A0)-(A5)} and
    \textbf{\upshape(B-ii)}. Then the function $\bar H_\alpha$ is
    convex.
\end{lemma}
\begin{proof}
For $p,\,q\in\R$, let $v_p$, $v_q$ be  solutions of
(\ref{eq:cell-alpha}) with $\lambda=\bar H_\alpha(p)$ and $\bar
H_\alpha(q)$ respectively. We also set
$$u_p(t,x)=v_p(t,x) + px - t\bar H_\alpha(p)$$
and define similarly $u_q$.

\paragraph{Step 1:  $u_p$ and $u_q$ are locally Lipschitz continuous.}
In this case, we have almost everywhere:
$$\left\{\begin{array}{l}
(u_p)_t + H_\alpha(t,x,(u_p)_x) = 0,\\
(u_q)_t + H_\alpha(t,x,(u_q)_x) = 0.\\
\end{array}\right.$$
For $\mu\in \left[0,1\right]$, let
$$\bar u = \mu u_p + (1-\mu)u_q.$$
By convexity, we get almost everywhere
\begin{equation}\label{eq::g5}
\bar u_t + H_\alpha(t,x,\bar u_x) \le 0.
\end{equation}
We claim that the convexity of $H_\alpha$ (in the gradient variable) implies
that $\bar u$ is a viscosity sub-solution.  To see it, we use an
argument of \cite[Proposition 5.1]{bcd}.  For $P=(t,x)$, we define a
mollifier $\rho_\delta(P)=\delta^{-2}\rho(\delta^{-1}P)$ and set
$$\bar u_\delta=\bar u\star \rho_\delta$$
Then by convexity, we get with $Q=(s,y)$:
$$(\bar u_\delta)_t + H_\alpha(P,(\bar u_\delta)_x) \le \int dQ\
\left\{H_\alpha(P,\bar u_x(Q))-H_\alpha(Q,\bar u_x(Q)\right\}\rho_\delta(P-Q).$$

The fact that $\bar u_x$ is locally bounded and the fact that $H_\alpha$ is
continuous imply that the right hand side goes to zero as $\delta \to
0$. We deduce (by stability of viscosity sub-solutions) that
(\ref{eq::g5}) holds true in the viscosity sense.  Then the comparison
principle implies that
\begin{equation}\label{eq::g10}
\mu \bar H_\alpha(p)+ (1-\mu) \bar H_\alpha(q) \ge \bar H_\alpha(\mu p + (1-\mu)q).
\end{equation}

\paragraph{Step 2: $u_p$ and $u_q$ are continuous.} We
proceed in two (sub)steps. 

\noindent \textsc{Step 2.1: the case of a single function $u$.}  We
first want to show that if $u=u_p$ is continuous and satisfies
(\ref{eq::g5}) almost everywhere, then $u$ is a viscosity
sub-solution. To this end, we will use the structural assumptions
satisfied by the Hamiltonian.  The ones that were useful to prove the
comparison principle will be also useful to prove the result we want.
Indeed, we will revisit the proof of the comparison principle. We also
use the fact that
\begin{equation}\label{eq::g6}
u(t,x)-px + t\bar H_\alpha(p) \quad \mbox{is bounded}.
\end{equation}

For $\nu>0$, we set
\[u^\nu(t,x)=\sup_{s\in \R} \left(u(s,x)-\frac{(t-s)^2}{2\nu}\right)=u(s_\nu,x)-\frac{(t-s_\nu)^2}{2\nu}.\]
As usual, we get from (\ref{eq::g6}) that 
\begin{equation}\label{eq::g8}
\left|t-s_\nu\right|  \le C\sqrt{\nu} \quad \mbox{with}\quad C=C(p,T)
\end{equation}
for $t \in (-T,T)$. In particular $ s_\nu \to t$ locally uniformly. 
If a test function $\varphi$ touches $u^\nu$ from above at some point $(t,x)$, then we have
$\displaystyle \varphi_t(t,x)=-\frac{t-s_\nu}{\nu}$ and
\begin{align}
\varphi_t(t,x)+ H_\alpha(t,x, \varphi_x(t,x)) &
\le H_\alpha(t,x, \varphi_x(t,x))-H_\alpha(s_\nu,x, \varphi_x(t,x))\nonumber \\
& \le \omega(\left|t-s_\nu\right|(1+\max (0, H_\alpha(s_\nu,x,
\varphi_x(t,x))))) \nonumber \\
& \le \omega\left(\frac{(t-s_\nu)^2}{\nu}+ \left|t-s_\nu\right|\right)
\label{eq::g7}
\end{align}
where we have used {\bf (A2)} in the third line.
The right hand side goes to zero as $\nu$ goes to zero since
\[\frac{(t-s_\nu)^2}{\nu} \to 0 \quad \mbox{locally uniformly w.r.t. $(t,x)$}\]
(recall $u$ is continuous). Indeed, this can be checked for $(t,x)$ replaced by $(t_\nu,x_\nu)$
because for any sequence $(t_\nu,s_\nu,x_\nu)\to (t,t,x)$, we have
$$u(t_\nu,x_\nu) \le u^\nu(t_\nu,x_\nu)=u(s_\nu,x_\nu) - \frac{(t_\nu-s_\nu)^2}{2\nu}$$
where the continuity of $u$ implies the result.
For a given $\nu>0$, we see that  (\ref{eq::g8}) and (\ref{eq::g7}) imply that
$$\left|\varphi_t \right|, \left|\varphi_x\right| \le C_{\nu,p}.$$
This implies in particular that $u^\nu$ is Lipschitz continuous, and then
$$u^\nu_t + H(t,x,u^\nu_x)\le o_\nu(1) \quad \mbox{a.e.}$$
where $o_\nu(1)$ is locally uniform with respect to $(t,x)$. \bigskip

\noindent {\sc Step 2.2: application.}
Applying Step 2.1, we get for $z=p,q$
$$(u^\nu_z)_t + H(t,x,(u^\nu_z)_x)\le o_\nu(1) \quad \mbox{a.e.}$$
where $o_\nu(1)$ is locally uniform with respect to $(t,x)$.
Step 1 implies that 
$$\bar u^\nu:= \mu u^\nu_p + (1-\mu)u^\nu_q$$
is a viscosity sub-solution of 
$$(\bar u^\nu)_t + H_\alpha(t,x,(\bar u^\nu)_x)\le o_\nu(1)$$
where $o_\nu(1)$ is locally uniform with respect to $(t,x)$.
In the limit $\nu\to 0$, we recover (by stability of sub-solutions) that $\bar u$ is a viscosity sub-solution, 
\textit{i.e.}  satisfies (\ref{eq::g5}) in the viscosity sense. This gives then the same conclusion as in Step 1.

\paragraph{Step 3: the general case.}
To cover the general case, we simply replace $u_p$ by $\tilde{u}_p$ which is the solution to the Cauchy problem
$$\left\{\begin{array}{l}
(\tilde{u}_p)_t + H_\alpha(t,x,(\tilde{u}_p)_x)=0,\quad \mbox{for}\quad (t,x)\in (0,+\infty)\times \R\\
\tilde{u}_p(0,x)=px,
\end{array}\right.$$
Then $\tilde{u}_p$ is continuous and satisfies $\left|\tilde{u}_p-u_p\right|\le C$.
Proceeding similarly with $\tilde{u}_q$ and using Step 2, we deduce
the desired inequality (\ref{eq::g10}). The proof is now complete.
\end{proof}
We finally prove the quasi-convexity of $\bar  H_\alpha$ under
assumption~\textbf{(B-i)}. 
\begin{lemma}[Quasi-convexity of $\bar H_\alpha$ under
    \textbf{(B-i)}]\label{lem:quasi-conv}
    Assume \textbf{\upshape (A0)-(A5)} and
    \textbf{\upshape(B-i)}. Then the function $\bar H_\alpha$ is
    quasi-convex.
\end{lemma}
\begin{proof}
  We reduce quasi-convexity to convexity by composing with an
  increasing function $\gamma$; notice that such a reduction was already used in optimization and  in partial differential
  equations, see for instance \cite{lions-convexe,ka}.

We first assume that $H_\alpha$ satisfies
\begin{equation}\label{eq::g11}
\left\{\begin{array}{l}
H_\alpha\in C^2,\\
D^2_{pp}H_\alpha(x,p^0_\alpha)>0,\\
D_p H_\alpha(x,p)<0 \quad \mbox{for}\quad p\in (-\infty,p^0_\alpha),\\
D_p H_\alpha(x,p)>0 \quad \mbox{for}\quad p\in (p^0_\alpha, +\infty),\\
H_\alpha(x,p)\to +\infty \quad \mbox{as}\quad |p|\to +\infty \quad
\mbox{uniformly w.r.t. } x\in \R.
\end{array}\right.
\end{equation}
For a function $\gamma$ such that
\[\gamma \quad \text{is convex},\quad \gamma \in C^2(\R) \quad \text{and}\quad \gamma' \ge \delta_0>0\]
we have 
\[D^2_{pp} (\gamma \circ H_\alpha) >0\]
if and only if
\begin{equation}\label{eq:cns}
(\ln \gamma')'(\lambda) > -\frac{D^2_{pp}H_\alpha(x,p)}{(D_p H_\alpha(x,p))^2}\quad \mbox{for}\quad 
p=\pi_\alpha^\pm(x,\lambda)\quad \text{and}\quad \lambda\ge H_\alpha(x,p)
\end{equation}
where $\pi_\alpha^\pm(x,\lambda)$ is the only real number $r$ such that
$\pm r \ge 0$ and $H_\alpha (x,r) = \lambda$. 
Because $D^2_{pp}H_\alpha(x,p^0_\alpha)>0$, we see that the right hand side is negative for
$\lambda$ close enough to $H_\alpha(x,p^0_\alpha)$ and it is indeed possible to
construct such a function $\gamma$. 

In view of Remark~\ref{rem:time-indep}, we can construct a solution of $\delta v^\delta + \gamma\circ H_\alpha(x,p+v^\delta_x)=0$
with $-\delta v^\delta \to \overline{\gamma\circ H_\alpha}(p)$ as $\delta\to 0$, and a solution of
$$\gamma\circ H_\alpha(x,p+v_x)=\overline{\gamma\circ H_\alpha}(p)$$
This shows that
\[ \bar H_\alpha  = \gamma^{-1} \circ \overline{\gamma \circ H_\alpha}.\]
Thanks to Lemmas~\ref{lem:coercive} and \ref{lem:conv}, we know that
$\overline{\gamma \circ H_\alpha}$ is coercive and convex. Hence $\bar
H_\alpha$ is quasi-convex.

If now $H_\alpha$ does not satisfies (\ref{eq::g11}), then for all
$\varepsilon>0$, there exists $H_\alpha^\varepsilon \in C^2$ such that 
\[ \begin{cases} 
(D^2_{pp}H_\alpha^\varepsilon)(x,p^0_\alpha)>0  \\
D_p H_\alpha^\eps(x,p)<0 \quad \mbox{for}\quad p\in (-\infty,p^0_\alpha),\\
D_p H_\alpha^\eps(x,p)>0 \quad \mbox{for}\quad p\in (p^0_\alpha, +\infty),\\
|H_\alpha^\varepsilon - H_\alpha| < \varepsilon.
\end{cases}\]
Then we can argue as in the proof of continuity of $\bar H_\alpha$
and deduce that 
\[ \bar H_\alpha (p ) = \lim_{\varepsilon \to 0} \bar H_\alpha^\varepsilon
(p). \]
Moreover, the previous case implies that $\bar H_\alpha^\varepsilon$ is
quasi-convex. Hence, so is $\bar H_\alpha$. The proof of the lemma is
now complete.
\end{proof}
\begin{proof}[Proof of Proposition~\ref{prop:quasi-conv}]
Combine Lemmas~\ref{lem:exis-base}, \ref{lem:unique},
\ref{lem:coercive}, \ref{lem:conv} and \ref{lem:quasi-conv}.
\end{proof}

\section{Truncated cell problems}
\label{sec:trunc}

We consider the following problem: find $\lambda_\rho\in \R$ and $w$
such that
\begin{equation}\label{eq:cell-trunc}
\left\{\begin{array}{ll}
w_t + H(t,x,w_x)=\lambda_\rho, &  (t,x)\in \R\times (-\rho,\rho),\\
w_t + H^-(t,x,w_x)=\lambda_\rho, & (t,x)\in \R\times \left\{-\rho\right\},\\
w_t + H^+(t,x,w_x)=\lambda_\rho, & (t,x)\in \R\times
\left\{\rho\right\}, \medskip\\
w \text{ is $1$-periodic w.r.t. } t.
\end{array}\right.
\end{equation}
Even if our approach is different, we borrow here an idea from
\cite{at} by truncating the domain and by considering correctors in
$[-\rho,\rho]$ with $\rho \to +\infty$.

\subsection{A comparison principle}

\begin{proposition}[Comparison principle for a mixed boundary value problem]\label{prop:comp}
Let  $\rho_2>\rho_1>\rho_0$ and $\lambda\in \R$ and $v$ be a
super-solution of the following boundary value problem
\begin{equation}\label{eq:super-comp}
\left\{\begin{array}{ll}
v_t + H(t,x,v_x) \ge \lambda & \text{ for } (t,x)\in \R\times (\rho_1,\rho_2),\\
v_t + H^+(t,x,v_x) \ge \lambda & \text{ for }(t,x)\in \R\times \left\{\rho_2\right\},\\
v(t,x) \ge  U_0(t) & \text{ for }(t,x)\in \R\times
\left\{\rho_1\right\}, \medskip\\
v \text{ is $1$-periodic w.r.t. } t
\end{array}\right.
\end{equation}
where $U_0$ is continuous and for $\eps_0>0$ and $u$ be a sub-solution
of the following one
\begin{equation}\label{eq:sub-comp}
\left\{\begin{array}{ll}
u_t + H(t,x,u_x)\le \lambda -\eps_0 & \text{ for }  (t,x)\in \R\times (\rho_1,\rho_2),\\
u_t + H^+(t,x,u_x)\le \lambda -\eps_0 &\text{ for } (t,x)\in \R\times \left\{\rho_2\right\},\\
u(t,x) \le  U_0(t) & \text{ for } (t,x)\in \R\times
\left\{\rho_1\right\},\medskip\\
u \text{ is $1$-periodic w.r.t. } t.
\end{array}\right.
\end{equation}
Then $u \le v$ in $\R \times [\rho_1,\rho_2]$. 
\end{proposition}
\begin{remark}
A similar result holds true if the Dirichlet condition is imposed at
$x= \rho_2$ and  junction conditions
\begin{align*}
v_t + H^-(t,x,v_x) &\ge \lambda \quad \quad \quad \text{ at } x = \rho_1 \\
u_t + H^-(t,x,u_x)&\le \lambda -\eps_0 \quad \; \text{ at } x = \rho_1
\end{align*}
are imposed at $x = \rho_1$. 
\end{remark}
The proof of Proposition~\ref{prop:comp} is very similar to (in fact
simpler than) the proof of the comparison principle for
Hamilton-Jacobi equations on networks contained in \cite{im}. The main
difference lies in the fact that in our case, $u$ and $v$ are global
in time and the space domain is bounded. A sketch of the proof is
provided in Appendix shedding some light on the main differences.
Here the parameter $\varepsilon_0>0$ in (\ref{eq:sub-comp}) is used in place of
the standard correction term $-\eta/(T-t)$ for a Cauchy problem.

\subsection{Correctors on truncated domains}

\begin{proposition}[Existence and properties of a corrector on a
    truncated domain]\label{prop:cor-trunc} There exists a unique
  $\lambda_\rho \in \R$ such that there exists a solution $w^\rho=w$
  of \eqref{eq:cell-trunc}. Moreover, there exists a constant $C>0$
  independent of $\rho\in (\rho_0,+\infty)$ and a function $m^\rho
  \colon [-\rho,\rho] \to \R$ such that
\begin{equation}\label{eq:trunc-cor}
\left\{\begin{array}{ll}
\left|\lambda_\rho\right| \le C,\\
\left|m^\rho(x)-m^\rho(y)\right|\le C \left|x-y\right| & \text{ for }  x,y\in [-\rho,\rho],\\
\left|w^\rho(t,x)-m^\rho(x)\right|\le C  & \text{ for } (t,x)\in \R\times [-\rho,\rho].
\end{array}\right.
\end{equation}
\end{proposition}
\begin{proof}
In order to construct a corrector on the truncated domain, we proceed
classically by considering
\begin{equation}\label{eq:cell-trunc-approx}
\left\{\begin{array}{ll}
\delta w^\delta + w^\delta_t + H(t,x,w^\delta_x)= 0, &  (t,x)\in \R\times (-\rho,\rho)\,,\\
\delta w^\delta +w^\delta_t + H^-(t,x,w^\delta_x)= 0, & (t,x)\in \R\times \left\{-\rho\right\},\\
\delta w^\delta +w^\delta_t + H^+(t,x,w^\delta_x)=0, & (t,x)\in
\R\times \left\{\rho\right\}, \medskip\\
w^\delta \text{ is $1$-periodic w.r.t. } t.
\end{array}\right.
\end{equation}
A discontinuous viscosity solution of \eqref{eq:cell-trunc-approx} is
constructed by Perron's method (in the class of $1$-periodic functions
with respect to time) since $\pm \delta^{-1} C$ are trivial
super-/sub-solutions if $C$ is chosen as follows
\[ C = \sup_{t \in \R,\ x \in \R} |H(t,x,0)|.\]
In particular, the solution $w^\delta$ satisfies by construction
\begin{equation}\label{eq:lli}
 |w^\delta | \le \frac{C}\delta.
\end{equation}
We next consider 
\[ m^\delta (x) = \sup_{t \in \R} (w^\delta)^* (t,x).\]
We remark that the supremum is reached since $w^\delta$ is periodic
with respect to time; we also remark that $m^\delta$ is a viscosity
sub-solution of 
\[ H(t(x),x,m^\delta_x)\le C, \quad  x \in  (-\rho,\rho) \]
(for some function $t(x)$). In view of \textbf{(A3)}, we conclude that $m^\delta$ is globally Lipschitz
continuous and
\begin{equation}
\label{eq:grad-md} 
|m^\delta_x | \le C 
\end{equation}
for some constant $C$ which still only depends on $H$. 
Assumption \textbf{(A3)} also implies that,
\[ w^\delta_t \le C \]
(with $C$ only depending on $H$). In particular, the comparison
principle implies that for all $t \in \R$, $x \in (-\rho,\rho)$ and $h
\ge 0$, 
\[ w^\delta (t+h,x) \le w^\delta (t,x) + Ch.\]
Combining this information with the periodicity of $w^\delta$ with
respect to $t$, we conclude that for $t \in \R$ and $x \in (-\rho,\rho)$,
\[ |w^\delta(t,x) - m^\delta (x)| \le C.\]
In particular,
\[ |w^\delta (t,x) - w^\delta (0,0)| \le C.\]
We then consider 
\[ \overline w = \limsup_\delta{}^* (w^\delta - w^\delta (0,0)) \quad 
\text{ and } \quad \underline w = \liminf_\delta{}_* (w^\delta -
w^\delta (0,0)).\]
We next remark that \eqref{eq:lli} and \eqref{eq:grad-md} imply that there exists
$\delta_n \to 0$ such that 
\begin{align*}
 m^{\delta_n} - m^{\delta_n} (0) \to m^\rho \quad \text{ as } n \to
 +\infty \\
\delta_n w^{\delta_n} (0,0) \to - \lambda_\rho \quad \text{ as } n \to +\infty
\end{align*}
(the first convergence being  locally uniform). In particular, $\lambda$,
$\overline w$, $\underline w$ and $m^\rho$ satisfies
\begin{align*}
| \lambda_\rho | \le C \\
|\overline w - m^\rho | \le C \\
|\underline w -m^\rho | \le C\\
|m^\rho_x | \le C.
\end{align*}
Discontinuous stability of viscosity solutions of Hamilton-Jacobi
equations imply that $\overline w-2C$ and $\underline w$ are respectively
a sub-solution and a super-solution of \eqref{eq:cell-trunc}  and 
\[ \overline w - 2 C \le \underline w .\]
Perron's method is used once again in order to construct a solution
$w^\rho$ of \eqref{eq:cell-trunc} which is $1$-periodic with respect
to time. In view of the previous estimates, $\lambda_\rho$, $m^\rho$
and $w^\rho$ satisfy \eqref{eq:trunc-cor}. Proving the uniqueness of
$\lambda_\rho$ is classical so we skip it. The proof of the
proposition is now complete.
\end{proof}

\begin{proposition}[First definition of the effective flux limiter]\label{prop:bar-A}
The map $\rho\mapsto \lambda_\rho$ is non-decreasing and bounded in $(0,+\infty)$.
In particular, 
\[ \bar A = \lim_{\rho \to +\infty} \lambda_\rho \]
exists and $\bar A \ge \lambda_\rho$ for all $\rho >0$. 
\end{proposition}
\begin{proof}
  For $\rho'>\rho>0$, we see that the restriction of $w^{\rho'}$ to
  $\left[-\rho,\rho\right]$ is a sub-solution, as a consequence of
  \cite[Proposition 2.19]{im}.  The boundedness of the map follows
  from Proposition~\ref{prop:cor-trunc}. The proof is thus complete.
\end{proof}
We next prove that we can control $w^\rho$ from below  under
appropriate assumptions on $\bar A$. 
\begin{proposition}[Control of slopes on a
    truncated domain]\label{prop:slopes-trunc}  
    Assume first that $\bar A > \min \bar H_R$. Then for all
    $\delta>0$, there exists $\rho_\delta>0$ and $C_\delta >0$
    (independent on $\rho$) such that for $x \ge
    \rho_\delta$ and $h \ge 0$,
\begin{equation}\label{eq:estim-slope}
 w^\rho (t,x+h) -w^\rho (t,x) \ge (\bar p_R-\delta) h - C_\delta.
\end{equation}
If now we assume that $\bar A > \min \bar H_L$, then  for $x\le
-\rho_\delta$ and $h\ge 0$, 
\begin{equation}\label{eq:estim-slope-bis}
 w^\rho (t,x-h) -w^\rho (t,x) \ge (-\bar p_L-\delta) h - C_\delta
\end{equation}
for some $\rho_\delta>0$ and $C_\delta >0$ as above.
\end{proposition}
\begin{proof}
We only prove \eqref{eq:estim-slope} since the proof of
\eqref{eq:estim-slope-bis} follows along the same lines. Let $\delta
>0$. In view of \textbf{(A5)}, we know that there exists $\rho_\delta$ such
that
\begin{equation}\label{eq:error-H}
 |H(t,x,p) - H_R (t,x,p)| \le \delta \quad \text{ for } \quad x \ge
\rho_\delta.
\end{equation}
Assume that $\bar A > \min \bar H_R$. Then Proposition
\ref{prop:quasi-conv} implies that we can pick $p^\delta_R$ such that
\[ \bar H_R (p^\delta_R) =\bar H_R^+ (p^\delta_R) = \lambda_\rho - 2
\delta  \]
for $\rho \ge \rho_0$ and $\delta \le \delta_0$, by choosing $\rho_0$
large enough and $\delta_0$ small enough. 

We now fix $\rho \ge \rho_\delta$ and $x_0\in \left[\rho_\delta,\rho\right]$. In view of
Proposition~\ref{prop:quasi-conv} applied to $p = p_R^\delta$, we know
that there exists a corrector $v_R$ solving \eqref{eq:cell-alpha} with
$\alpha =R$. Since it is $\Z^2$-periodic, it is bounded and $w_R
= p_R^\delta x + v_R (t,x)$ solves
\[(w_R)_t + H_R(t,x,(w_R)_x)=  \lambda_\rho -2\delta, \quad  (t,x)\in \R\times \R.\]
In particular, the restriction of $w_R$ to $[\rho_\delta,\rho]$
satisfies (see  \cite[Proposition 2.19]{im}),
\[\left\{\begin{array}{ll}
(w_R)_t + H_R(t,x,(w_R)_x)\le \lambda_\rho -2\delta & \text{ for }  (t,x)\in \R\times (\rho_\delta,\rho),\\
(w_R)_t + H^+_R(t,x,(w_R)_x)\le \lambda_\rho -2\delta &\text{ for } (t,x)\in \R\times \left\{\rho\right\}.
\end{array}\right.\]
In view of \eqref{eq:error-H}, this implies
\[\left\{\begin{array}{ll}
(w_R)_t + H(t,x,(w_R)_x)\le \lambda_\rho -\delta  & \text{ for }  (t,x)\in \R\times (\rho_\delta,\rho),\\
(w_R)_t + H^+(t,x,(w_R)_x)\le \lambda_\rho -\delta  &\text{ for } (t,x)\in \R\times \left\{\rho\right\}.
\end{array}\right.\]
Now we remark that $v = w^\rho - w^\rho (0,x_0)$ and $u = w_R -w_R
(0,x_0) -2C - 2 \|v_R\|_\infty$ satisfies 
\[ v(t,x_0) \ge - 2C \ge u (t,x_0)  \]
where $C$ is given by \eqref{eq:trunc-cor}. Thanks to the comparison
principle from Proposition~\ref{prop:comp}, we thus get for $x \in [x_0,\rho]$,
\[ w^\rho (t,x) -w^\rho (t,x_0) \ge p_R^\delta (x-x_0) - C_\delta\]
where $C_\delta$ is a large constant which does not depend on $\rho$. In
particular, we get \eqref{eq:estim-slope}, reducing $\delta$ if necessary.
\end{proof}

\subsection{Construction of global correctors}

We now state and prove a result which implies
Theorem~\ref{thm:corrector-simple} stated in the introduction. 
\begin{theorem}[Existence of a global corrector for the junction]\label{thm:corrector}
Assume \textbf{\upshape (A0)-(A5)} and either \textbf{\upshape (B-i)}
or \textbf{\upshape (B-ii)}.
\begin{enumerate}[\upshape i)]
\item {\upshape(General properties)} There exists a solution $w$ of
  \eqref{eq:cell} with $\lambda = \bar A$ such that for all $(t,x) \in
  \R^2$,
\begin{equation}\label{eq:w-x-osc}
\left|w(t,x)-m(x)\right|\le C  
\end{equation}
for some globally Lipschitz continuous function $m$, and
\[\bar A\ge A_0.\]
\item {\upshape(Bound from below at infinity)}
If $\bar A > \max_{\alpha=L,R}\left(\min \bar H_\alpha\right)$, then there exists $\delta_0>0$ such that for every $\delta\in
(0,\delta_0)$, there exists $\rho_\delta>\rho_0$ such that $w$
satisfies
\begin{equation}
\label{eq:w-lb-slope}
\begin{cases}
w(t,x+h)-w(t,x)\ge (\bar p_R -\delta) h -C_\delta & \quad \text{ for } x\ge
\rho_\delta \quad \text{and}\quad h\ge 0 , \\
w(t,x-h)-w(t,x)\ge (-\bar p_L -\delta) h -C_\delta & \quad \text{ for } x\le 
-\rho_\delta \quad \text{and}\quad h\ge 0.
\end{cases}
\end{equation}
The first line of (\ref{eq:w-lb-slope}) also holds if we have only $\bar A> \min
\bar H_R$, while the second line of (\ref{eq:w-lb-slope}) also holds 
if we have only $\bar A>  \min \bar H_L$.
\item {\upshape (Rescaling $w$)}
For $\eps>0$, we set
\[w^\eps(t,x)=\eps w(\eps^{-1}t,\eps^{-1}x).\]
Then (along a subsequence $\eps_n\to 0$), we have that 
$w^\eps$ converges locally uniformly towards a function $W=W(x)$ which satisfies
\begin{equation}\label{eq:W-estim}
\left\{\begin{array}{ll}
\left|W(x)-W(y)\right|\le C \left|x-y\right| &\quad \text{for all}\quad x,y\in \R,\\
\bar{H}_R(W_x)=\bar A \quad \text{and}\quad \hat{p}_R \ge W_x\ge \bar p_R &\quad \text{ for } x\in (0,+\infty),\\
\bar{H}_L(W_x)=\bar A \quad \text{and}\quad \hat{p}_L \le  W_x\le \bar p_L &\quad \text{ for } x\in  (-\infty,0).
\end{array}\right.
\end{equation}
In particular, we have  $W(0)=0$ and
\begin{equation}\label{eq:W-estim-2-bis}
\hat p_R x 1_{\left\{x>0\right\}} + \hat p_L x
1_{\left\{x<0\right\}}\ge W(x)\ge \bar p_R x
1_{\left\{x>0\right\}} + \bar p_L x 1_{\left\{x<0\right\}}.
\end{equation}
\end{enumerate}
\end{theorem}
\begin{proof}
We consider (up to some subsequence)
\[ \overline w = \limsup_{\rho\to +\infty}{}^* (w^\rho - w^\rho (0,0)),\quad \underline w = \liminf_{\rho\to +\infty}{}_* (w^\rho -
w^\rho (0,0))\quad  \text{ and } \quad  m=\lim_{\rho\to +\infty} (m^\rho-m^\rho(0)).\]
We derive from \eqref{eq:trunc-cor} that $\underline w$ and $\overline
w$ are finite and
\[ m- C \le \underline w \le \overline w \le m+C. \]
Moreover, discontinuous stability of viscosity solutions imply that
$\overline w-2C$ and $\underline w$ are respectively a sub-solution and a
super-solution of \eqref{eq:cell} with $\lambda = \bar A$ (recall
Proposition~\ref{prop:bar-A}). Hence, a discontinuous viscosity
solution $w$ of \eqref{eq:cell} can be constructed by Perron's
method (in the class of functions that are $1$-periodic with respect
to time). 

Using again \eqref{eq:trunc-cor}, $w$ and $m$ satisfy
\eqref{eq:w-x-osc}.  We also get \eqref{eq:w-lb-slope} from
Proposition~\ref{prop:slopes-trunc} (use \eqref{eq:trunc-cor} and
pass to the limit with $m$ instead of $w$ if necessary).

We now study $w^\eps(t,x)=\varepsilon
w(\varepsilon^{-1}t,\varepsilon^{-1}x)$. Remark that
\eqref{eq:trunc-cor} implies in particular that
\[ w^\eps (t,x) = \eps m(\eps^{-1} x) + O (\eps) .\] In particular, we
can find a sequence $\eps_n \to 0$ such that 
\[ w^{\eps_n} (t,x) \to W(x) \quad \text{ locally uniformly as } n \to
+ \infty,\] with $W(0)=0$.  Arguing as in the proof of convergence
away from the junction point (see the case $\bar x \neq 0$ in
Appendix), we deduce that $W$ satisfies
\begin{align*}
\bar H_R (W_x ) = \bar A  & \text{ for } x >0, \\
\bar H_L (W_x ) = \bar A & \text{ for } x<0.
\end{align*}
We also deduce from \eqref{eq:w-lb-slope} that for all $\delta >0$ and $x>0$,
\[ W_x \ge \bar p_R - \delta \]
in the case where $\bar A > \min \bar H_R$. Assume now that $\bar A =
\min \bar H_R$. This implies that 
\[\bar  p_R \le W_x \le \hat p_R\]
and, in all cases, we thus get (\ref{eq:W-estim-2-bis}) for $x>0$.

Similarly, we can prove for $x<0$ that 
\[\hat p_L\le W_x \le \bar p_L \]
and the proof of \eqref{eq:W-estim} of is achieved.  This implies
(\ref{eq:W-estim-2-bis}).  The proof of Theorem~\ref{thm:corrector} is now
complete.
\end{proof}

\subsection{Proof of Theorem~\ref{thm:bar-A}}

\begin{proof}[Proof of Theorem~\ref{thm:bar-A}]
Let $\bar A$ denote the limit of $A_\rho$ (see
Proposition~\ref{prop:bar-A}). We want to prove that $\bar A = \inf E$
where we recall that 
\[ E= \{ \lambda \in \R: \exists \text{$w$ sub-solution of
  \eqref{eq:cell}} \}. \]
In view of \eqref{eq:cell}, sub-solutions are assumed to be periodic with 
respect to time; we will see that they also automatically satisfy some
growth conditions at infinity, see \eqref{eq:growth} below. 

We argue by contradiction by assuming that there exist $\lambda <
\bar A$ and  a sub-solution $w_\lambda$ of \eqref{eq:cell}.
The function 
\[ m_\lambda (x) = \sup_{t \in \R} (w_\lambda)^* (t,x) \]
satisfies 
\[ H (t(x),x,(m_\lambda)_x ) \le C \]
(for some function $t(x)$). 
Assumption~\textbf{(A3)} implies that $m_\lambda$ is globally Lipschitz
continuous. Moreover, since $w_\lambda$ is $1$-periodic w.r.t. time and $(w_\lambda)_t
\le C$, then 
\[ |w_\lambda(t,x) -m_\lambda(x)  | \le C.\]
Hence 
\[w_\lambda^\eps (t,x) = \eps w_\lambda (\eps^{-1}t,\eps^{-1}x) \]
has a limit $W^\lambda$ which satisfies 
\[ \bar H_R (W^\lambda_x) \le \lambda \quad \text{ for } x >0 .\]
In particular, for $x>0$, 
\[ W^\lambda_x \le \hat p^\lambda_R := \max \{ p \in \R: \bar H_R (p) =
\lambda \} < \bar p_R\]
where $\bar p_R$ is defined in \eqref{def:pR}.
Similarly, 
\[ W^\lambda_x \ge \hat p^\lambda_L := \min \{ p \in \R: \bar H_L(p)=
\lambda \} > \bar p_L\]
with $\bar p_L$ defined in \eqref{def:pL}. Those two inequalities
imply in particular that for all $\delta>0$, there exists
$\tilde{C}_\delta$ such that 
\begin{equation}
\label{eq:growth}
 w_\lambda (t,x) \le \begin{cases}
(\hat p^\lambda_R + \delta) x + \tilde{C}_\delta & \text{ for } x >0, \\
(\hat p^\lambda_L + \delta) x + \tilde{C}_\delta & \text{ for } x <0.
\end{cases}
\end{equation}
In particular, 
\[ w_\lambda  < w \text{ for } |x| \ge R \]
if $\delta$ is small enough and $R$ is large enough. In particular,
\[ w_\lambda < w + C_R \text{ for } x \in \R.\]
Remark finally that $u (t,x) = w(t,x) + C_R - \bar A t$ is a solution and $u_\lambda (t,x)
= w_\lambda (t,x) - \lambda t$ is a sub-solution of \eqref{eq:hj-eps}
with $\eps = 1$ and $u_\lambda (0,x) \le u (0,x)$. Hence the comparison
principle implies that 
\[ w_\lambda(t,x) - \lambda t \le w(t,x) - \bar A t + C_R.\]
Dividing by $t$ and letting $t$ go to $+\infty$, we get the following
contradiction
\[ \bar A \le \lambda.\]
The proof is now complete. 
\end{proof}

\section{Proof of Theorem~\ref{thm:conv-time}}
\label{sec:extension}

This section is devoted to the proof of
Theorem~\ref{thm:conv-time}. As pointed out in
Remark~\ref{rem:meaning} above, the notion of solutions for
\eqref{eq:hj-eps} has to be first made precise because
the Hamiltonian is discontinuous with respect to time.

\paragraph{Notion of solutions for \eqref{eq:hj-eps}.}

For $\varepsilon=1$, a function $u$ is a {\it
  solution} of \eqref{eq:hj-eps} if it is globally Lipschitz
continuous (in space and time) and if it solves successively the
Cauchy problems on time intervals $[\tau_i+k,\tau_{i+1}+k)$ for
$i=0,\dots, K$ and $k\in \N$.  

Because of this definition (approach), we have to show that if the
initial datum $u_0$ is globally Lipschitz continuous, then the solution
to the successive Cauchy problems is also globally Lipschitz
continuous (which of course insures its uniqueness from the classical
comparison principle). See Lemma~\ref{lem:glob-lip} below.

\begin{proof}[Proof of Theorem~\ref{thm:conv-time} i)] In view of the
  proof of Theorem~\ref{thm:conv}, the reader can check that it is
  enough to get a global Lipschitz bound on the solution $u^\eps$ and
  to construct a global corrector in this new framework. The proof of
  these two facts is postponed, see Lemmas~\ref{lem:glob-lip} and
  \ref{lem:cor-ok} following this proof. Notice that half-relaxed
  limits are not necessary anymore and that the reasoning can be
  completed by considering locally converging subsequences of
  $\{u^\eps\}_\eps$. Notice also that the perturbed test
    function method of Evans \cite{evans} still works.  As usual, if the viscosity
    sub-solution inequality is not satisfied at the limit, this
    implies that the perturbed test function is a super-solution
    except at times $\varepsilon\left(\Z +
      \left\{\tau_0,\dots,\tau_K\right\}\right)$. Still a
    localized comparison principle in each slice of times for each
    Cauchy problem is sufficient to conclude.
\end{proof}
\begin{lemma}[Global Lipschitz bound]\label{lem:glob-lip}
The function $u^\eps$ is equi-Lipschitz continuous with respect to
time and space. 
\end{lemma}
\begin{proof}
Remark that it is enough to get the result for $\eps =1$ since 
$u(t,x) = \eps^{-1} u^\eps (\eps t,\eps x)$ satisfies the equation
with $\eps =1$ and the initial condition
\[ u^\eps_0 (x) = \frac1\eps U^\eps_0 (\eps x) \]  
is equi-Lipschitz continuous. For the sake of clarity, we drop the $\eps$
superscript in $u^\eps_0$ and simply write $u_0$. 

We first derive bounds on the time interval
$[\tau_0,\tau_1)=[0,\tau_1)$.  In order to do so, we assume that the
initial data satisfies $|(u_0)_x|\le L$. Then as usual, there is a
constant $C>0$ such that
$$u^\pm(t,x)=u_0(x)\pm C t$$
are super-/sub-solutions of \eqref{eq:hj-eps}-\eqref{eq:ic-bis} with $H$ given by \textbf{(C1)}
with for instance
\begin{equation}\label{eq::g30}
C:=\max\left( \max_{\alpha=1,\dots,N}\|a_\alpha\|_\infty, \; \max_{\alpha=0,\dots,N}\left(\max_{|p|\le L} |\bar H_\alpha(p)|\right)\right).
\end{equation}
Let $u$ be the standard (continuous) viscosity solution of
(\ref{eq:hj-eps}) on the time interval $(0,\tau_1)$ with initial
data given by $u_0$ (recall that $\eps =1$).  Then for any $h>0$ small
enough, we have $-Ch\le u(h,x)-u(0,x)\le Ch$. The comparison principle
implies for $t\in (0,\tau_1-h)$
$$-Ch\le u(t+h,x)-u(t,x)\le Ch$$
which shows the Lipschitz bound in time, on the time interval $[0,\tau_1)$:
\begin{equation}\label{eq::g31}
|u_t|\le C.
\end{equation}
From the Hamilton-Jacobi equation, we now deduce the following
Lipschitz bound in space on the time interval $(0,\tau_1)$:
\begin{equation}\label{eq::g32}
|\bar H_\alpha(u_x(t,\cdot))|_{L^\infty(b_{\alpha},b_{\alpha+1})}\quad \le C\quad \mbox{for}\quad \alpha=0,\dots,N.
\end{equation}

We can now derive bounds on the time interval $[\tau_1,\tau_2)$ as
follows.  We deduce first that (\ref{eq::g32}) still holds
true at time $t=\tau_1$.  Combined with our definition (\ref{eq::g30})
of the constant $C$, we also deduce that
$$v^\pm(t,x)=u(\tau_1,x) \pm C(t-\tau_1)$$
are sub/super-solutions of (\ref{eq:hj-homog}) for $t\in
(\tau_1,\tau_2)$ where $H$ is given by \textbf{(C1)}.
Reasoning as above, we get bounds (\ref{eq::g31}) and
(\ref{eq::g32}) on the time interval $[\tau_1,\tau_2)$.

Such a reasoning can be used iteratively to get the Lipschitz bounds
(\ref{eq::g31}) and (\ref{eq::g32})  for $t\in
[0,+\infty)$. The proof of the lemma is now complete.
\end{proof}
\begin{lemma}\label{lem:cor-ok}
The conclusion of Theorem~\ref{thm:corrector} still holds true in this
new framework. 
\end{lemma}
\begin{proof}
The proof proceeds in several steps. 

\paragraph{Step 1. Construction of a time periodic corrector $w^\rho$
  on $[-\rho,\rho]$.} \medskip
We first construct a Lipschitz corrector on a truncated domain. In
order to do so, we proceed in several steps. \medskip

\noindent {\sc Step 1.1. First Cauchy problem on $(0,+\infty)$.}
The method presented in the proof of Proposition \ref{prop:cor-trunc},
using a term $\delta w^\delta$ has the inconvenience that it would not
clearly provide a Lipschitz solution.  In order to stick to our notion
of globally Lipschitz solutions, we simply solve the Cauchy problem
for $\rho>\rho_0 :=\max_{\alpha=1,\dots,N}|b_\alpha|$:
\begin{equation}\label{eq::g33}
\left\{\begin{array}{ll}
w^\rho_t + H(t,x,w^\rho_x)=0 &\quad \mbox{on}\quad (0,+\infty)\times
(-\rho,\rho)\, ,\\
w^\rho_t + \bar H^-_N (w^\rho_x)=0 &\quad \mbox{on}\quad (0,+\infty)\times \left\{-\rho\right\},\\
w^\rho_t + \bar H^+_0 (w^\rho_x)=0 &\quad \mbox{on}\quad (0,+\infty)\times \left\{\rho\right\},\\
w^\rho(0,x)=0 &\quad \mbox{for}\quad x\in [-\rho,\rho]\, .
\end{array}\right.
\end{equation}
As in the proof of the previous lemma, we get global Lipschitz bounds
with a constant $C$ (independent on $\rho>0$ and independent on the
distances $\ell_\alpha = b_{\alpha+1}-b_\alpha$):
\begin{equation}\label{eq::g34}
|w^\rho_t|,\quad  |\bar
H_\alpha(w^\rho_x(t,\cdot))|_{L^\infty((b_{\alpha},b_{\alpha+1})\cap
  (-\rho,\rho))} \quad \le C,\quad \mbox{for}\quad \alpha=0,\dots,N.
\end{equation}
Arguing as in \cite{im0} for instance, we deduce
that there exists a real number $\lambda_\rho$ 
with
$$|\lambda_\rho|\le C$$
and a constant $C_0$ (that depends on $\rho$) such that we have
\begin{equation}\label{eq::g35}
|w^\rho(t,x)+\lambda_\rho t|\le C_0.
\end{equation} 
Details are given in Appendix for the reader's convenience. 
 \medskip

\noindent {\sc Step 1.2. Getting global sub and super-solutions.}
Let us now define  the following function (up to some subsequence
$k_n\to +\infty$):
$$w^\rho_\infty(t,x)=\lim_{k_n\to +\infty} \left(w^\rho(t+k_n,x)+\lambda_\rho k_n\right)$$
which still satisfies (\ref{eq::g34}) and (\ref{eq::g35}). Then we
also define the two functions
$$\overline{w}^\rho_\infty(t,x)=\inf_{k\in \Z} \left(w^\rho_\infty(t+k,x)+k\lambda_\rho\right),\quad 
\underline{w}^\rho_\infty(t,x)=\sup_{k\in \Z}
\left(w^\rho_\infty(t+k,x)+k\lambda_\rho\right).$$
They still satisfy
\eqref{eq::g34} and \eqref{eq::g35} and are respectively a super- and
a sub-solution of the problem in $\R\times [-\rho,\rho]$. They satisfy
moreover that $\overline{w}^\rho_\infty(t,x) +\lambda_\rho t$ and
$\underline{w}^\rho_\infty(t,x) +\lambda_\rho t$ are $1$-periodic in
time, which implies the following bounds
\[|\overline{w}^\rho_\infty(t,x) -\overline{w}^\rho_\infty(0,x)  
+\lambda_\rho t|\le C,\quad |\underline{w}^\rho_\infty(t,x) -\underline{w}^\rho_\infty(0,x)  +\lambda_\rho t|\le C.\]

\noindent \textsc{Step 1.3: A new Cauchy problem on $(0,+\infty)$ and construction of a time periodic solution.}
We note that $\overline{w}^\rho_\infty + 2C_0 \ge
\underline{w}^\rho_\infty$, and we now solve the Cauchy problem with
new initial data $\underline{w}^\rho_\infty(0,x)$ instead of the zero
initial data and call $\tilde{w}^\rho$ the solution of this new Cauchy
problem.  From the comparison principle, we get
$$\underline{w}^\rho_\infty \le \tilde{w}^\rho \le \overline{w}^\rho_\infty + 2C_0.$$
In particular, 
$$\tilde{w}^\rho(1,x)\ge \underline{w}^\rho_\infty(1,x) \ge \tilde{w}^\rho(0,x)-\lambda_\rho.$$
This implies, by comparison, that
\begin{equation}\label{eq::g36}
\tilde{w}^\rho(k+1,x) \ge \tilde{w}^\rho(k,x) - \lambda_\rho.
\end{equation}
Moreover $\tilde{w}^\rho$ still satisfies (\ref{eq::g34}) (indeed  with the same constant 
because, by construction, this is also the case for $\underline{w}^\rho_\infty$).
We now define (up to some subsequence $k_n\to +\infty$):
$$\tilde{w}^\rho_\infty(t,x)=\lim_{k_n\to +\infty} \left(\tilde{w}^\rho(t+k_n,x)+\lambda_\rho k_n\right)$$
which, because of (\ref{eq::g36}) and the fact that $\tilde{w}^\rho(t,x)+\lambda_\rho t$ is bounded, satisfies
$$\tilde{w}^\rho_\infty(k+1,x)+\lambda=\tilde{w}^\rho_\infty(k,x)$$
and then $\tilde{w}^\rho_\infty(t,x)+\lambda_\rho t$ is $1$-periodic in time. Moreover $\tilde{w}^\rho_\infty$ is still a solution of the Cauchy problem
 and satisfies (\ref{eq::g34}). We define
$$w^\rho:=\tilde{w}^\rho_\infty + \lambda_\rho t$$
which satisfies (\ref{eq:trunc-cor}) and then provides the analogue of
the function given in Proposition~\ref{prop:cor-trunc}.

\paragraph{Step 2. Contruction of $w$ on $\R$.}
The result of Theorem \ref{thm:corrector} still holds true for
$$w=\lim_{\rho\to +\infty} \left(w^\rho-w^\rho(0,0)\right)$$
which is globally Lipschitz continuous in space and time and satisfies
(\ref{eq::g34}) with $\rho=+\infty$, and
$$\bar A =\lim_{\rho\to +\infty} \lambda_\rho. \qedhere $$
\end{proof}
\begin{proof}[Proof of \eqref{eq:N1} from Theorem~\ref{thm:conv-time}] 
 We recall that $\bar H_L=\bar H_0$ and
  $\bar H_R=\bar H_1$ and set $a=a_1$ and (up to translation) $b_1=0$.

\paragraph{Step 1: The convex case: identification of $\bar A$.} ~ \\

\noindent \textsc{Step 1.1: A convex subcase.}
We first work in the particular case where both $\bar H_\alpha$ for
$\alpha=L,R$ are convex and given by the Legendre-Fenchel transform of
convex Lagrangians $L_\alpha$ which satisfy for some compact interval
$I_\alpha$:
\begin{equation}\label{eq::g45}
L_\alpha(p)= \left\{\begin{array}{ll}
\mbox{finite} &\quad \mbox{if} \quad q\in I_\alpha,\\
+\infty &\quad \mbox{if} \quad q\not\in I_\alpha.
\end{array}\right.
\end{equation}
Then it is known (see for instance the section on optimal control in \cite{im})
that the solution of \eqref{eq:hj-eps} on the time interval $[0,\varepsilon\tau_1)$,  is given by
\begin{equation}\label{eq::g40}
u^\varepsilon(t,x)=\inf_{y\in \R} \left(\inf_{X\in S_{0,y;t,x}}
  \left\{u^\varepsilon(0,X(0)) + \int_0^t
    L_\varepsilon(s,X(s),\dot{X}(s))\, ds \right\} \right) 
\end{equation}
with
$$L_\varepsilon(s,x,p)=\left\{\begin{array}{ll}
\bar H_L^*(p) &\quad \mbox{if}\quad x<0,\\
\bar H_R^*(p) &\quad \mbox{if}\quad x>0,\\
\min(-a(\varepsilon^{-1}s), \displaystyle \min_{\alpha=L,R} L_\alpha(0)) &\quad \mbox{if}\quad x=0,\\
\end{array}\right.$$
and for $s<t$, the following set of trajectories:
$$S_{s,y;t,x} = \left\{X\in \mbox{Lip}((s,t);\R), \quad X(s)=y,\quad X(t)=x\right\}.$$
Combining this formula with the other one on the time interval
$[\varepsilon\tau_1,\varepsilon \tau_2)$, 
and iterating on all necessary intervals, we get that 
(\ref{eq::g40}) is a representation formula of the solution $u^\varepsilon$ of \eqref{eq:hj-eps} for  all $t>0$.
We also know (see the section on optimal control in  \cite{im}), that the optimal trajectories from $(0,y)$ to $(t_0,x_0)$
intersect the axis $x=0$ at most on a time interval $[t_1^\varepsilon, t_2^\varepsilon]$ with $0\le t_1^\varepsilon\le t_2^\varepsilon\le t_0$.
If this interval is not empty, then we have $t^\varepsilon_i\to t^0_i$ for $i=1,2$ and we can easily pass to the limit in 
(\ref{eq::g40}). In general,  $u^\varepsilon$ converges to $u^0$ given by the formula
$$u^0(t,x)=\inf_{y\in \R} \left(\inf_{X\in S_{0,y;t,x}}
  \left\{u^0(0,X(0)) + \int_x^t  L_0(s,X(s),\dot{X}(s)) \, ds \right\}\right)$$
with
$$L_0(s,x,p)=\left\{\begin{array}{ll}
\bar H_L^*(p) &\quad \mbox{if}\quad x<0,\\
\bar H_R^*(p) &\quad \mbox{if}\quad x>0,\\
\min(-\langle a\rangle, \displaystyle \min_{\alpha=L,R} L_\alpha(0)) &\quad \mbox{if}\quad x=0,\\
\end{array}\right.$$
and from \cite{im} we see that $u^0$ is the unique solution of
\eqref{eq:hj-homog}-\eqref{eq:ic} with $\bar A = \langle a\rangle$.
\medskip

\noindent \textsc{Step 1.2: The general convex case.}  The general
case of convex Hamiltonians is recovered, because for Lipschitz
continuous initial data $u_0$, we know that the solution is globally
Lipschitz continuous. Therefore, we can always modify the Hamiltonians
$\bar H_\alpha$ outside some compact intervals such that the modified
Hamiltonians satisfy (\ref{eq::g45}).

\paragraph{Step 2: General quasi-convex Hamiltonians: identification
  of $\bar A$.} ~ \\

\noindent \textsc{Step 2.1: Sub-Solution inequality.} From Theorem 2.10 in
\cite{im}, we know that $w(t,0)$, as a function of time only,
satisfies in the viscosity sense
$$w_t(t,0) + a(t)\le \bar A \quad \mbox{for all}\quad
t\notin\bigcup_{i=1,\dots,K+1}\tau_i + \Z.$$
Using the $1$-periodicity in time of $w$, we see that the integration in time on one period implies:
\begin{equation}\label{eq::g66}
\langle a\rangle \le  \bar A.
\end{equation}

\noindent \textsc{Step 2.2: Super-solution inequality.}
Recall that $\bar A\ge \ \langle a\rangle\ \ge A_0:=\displaystyle
\max_{\alpha=L,R}\min (\bar H_\alpha)$. If $\bar A= A_0$, then
obviously, we get $\bar A= \langle a\rangle$. Hence, it remains to
treat the case $\bar A>A_0$. \medskip

\noindent \textsc{Step 2.3: Construction of a super-solution for
  $x\not=0$.}  Recall that $\bar p_R$ and $\bar p_L$ are defined in
\eqref{def:pR} and \eqref{def:pL} and the minimum of $\bar H_\alpha$
is reached for $\bar p^0_\alpha$, $\alpha = R,L$. Since $\bar A >
A_0$, there exists some $\delta>0$ such that
\begin{equation}\label{eq::g60}
\bar p_L +2\delta < \bar p^0_L\quad \quad \mbox{and}\quad \bar p^0_R<\bar p_R-2\delta.
\end{equation}
If $w$ denotes a global corrector given by Lemma \ref{lem:cor-ok} (or
Theorem~\ref{thm:corrector}), let us define
$$\underline{w}_R(t,x)=\inf_{h\ge 0} \left(w(t,x+h)-\bar p^0_R h\right)\quad \mbox{for}\quad x\ge 0,$$
and similarly
$$\underline{w}_L(t,x)=\inf_{h\ge 0} \left(w(t,x-h)+\bar p^0_L h\right)\quad \mbox{for}\quad x\le 0.$$
From (\ref{eq:w-lb-slope}) with $\rho_\delta=0$, we deduce that we have for some $\bar h\ge 0$
$$w(t,x)\ge \underline{w}_R(t,x)=w(t,x+\bar h)-\bar p^0_R \bar h\ge w(t,x)+(\bar p_R-\delta-\bar p^0_R)\bar h - C_\delta.$$
From (\ref{eq::g60}), this implies
\begin{equation}\label{eq::re1}
0\le \bar h\le C_\delta/\delta
\end{equation}
and using the fact that $w$ is globally Lipschitz continuous, we deduce that for $\alpha=R$:
\begin{equation}\label{eq::g61}
w\ge \underline{w}_\alpha\ge w -C_1.
\end{equation}
Moreover, by constrution (as an infimum of (globally Lipschitz continuous)
super-solutions), $\underline{w}_R$ is a (globally Lipschitz continuous)
super-solution of the problem in $\R\times (0,+\infty)$.  We also have
for $x=y+z$ with $z\ge 0$:
\begin{align*}
\underline{w}_R(t,x)-\underline{w}_R(t,y) &=w(t,x+\bar h)-\bar p^0_R \bar h - \underline{w}_R(t,y)\\
&\ge w(t,x+\bar h)-\bar p^0_R \bar h -\left(w(t,y+\bar h +z)-\bar p^0_R (\bar h + z)\right)\\
&\ge \bar p^0_R  z = \bar p^0_R (x-y)
\end{align*}
which shows that
\begin{equation}\label{eq::g62}
(\underline{w}_R)_x\ge \bar p^0_R.
\end{equation}
Similarly (and we can also use a symmetry argument to see it), we get
that $\underline{w}_L$ is a (globally Lipschitz continuous) super-solution in
$\R\times (-\infty,0)$, it satisfies (\ref{eq::g61}) with $\alpha =L$ and
\begin{equation}\label{eq::g63}
(\underline{w}_L)_x\le \bar p^0_L.
\end{equation}
We now define
\begin{equation}\label{eq::g70}
\underline{w}(t,x)=\left\{\begin{array}{ll}
\underline{w}_R(t,x) &\quad \mbox{if}\quad x>0,\\
\underline{w}_L(t,x) &\quad \mbox{if}\quad x<0,\\
\min(\underline{w}_L(t,0), \underline{w}_R(t,0))&\quad \mbox{if}\quad x=0
\end{array}\right.
\end{equation}
which by constrution is lower semi-continuous and satisfies
\eqref{eq::g61}, and is a super-solution for $x\not=0$.

\noindent \textsc{Step 2.4: Checking the super-solution property at $x=0$.}
Let $\varphi$ be a test function touching $\underline{w}$ from below
at $(t_0,0)$ with $t_0\notin\bigcup_{i=1,\dots,K+1}
  \tau_i +\Z$.  We want to check that
\begin{equation}\label{eq::g64}
\varphi_t(t_0,0) + F_{a(t_0)}(\varphi_x(t_0,0^-),\varphi_x(t_0,0^+))\ge \bar A.
\end{equation}
We may assume that
\[\underline{w}(t_0,0)=\underline{w}_R(t_0,0)\] since the case
$\underline{w}(t_0,0)=\underline{w}_L(t_0,0)$ is completely similar.
Let $\bar h \ge 0$ be such that
\[\underline{w}_R(t_0,0) =w(t_0,0+\bar h)-\bar
p^0_R \bar h.\]

We distinguish two cases. Assume first that $\bar h>0$.
Then we have for all $h\ge 0$
$$\varphi(t,0)\le w(t,0+h)-\bar p^0_R h$$
with equality for $(t,h)=(t_0,\bar h)$.
This implies the viscosity inequality
$$\varphi_t(t_0,0) + \bar H_R(\bar p^0_R) \geq \bar A$$
which implies (\ref{eq::g64}), because
$F_{a(t_0)}(\varphi_x(t_0,0^-),\varphi_x(t_0,0^+)) \geq a(t_0)\geq
A_0\geq \min \bar H_R = \bar H_R(\bar p^0_R)$.

Assume now that $\bar h=0$.
Then we have $\varphi\le \underline{w}\le w$ with equality at
$(t_0,0)$. This implies immediately (\ref{eq::g64}).
\medskip

\noindent \textsc{Step 2.5: Conclusion.}
We deduce that $\underline{w}$ is a super-solution on $\R\times \R$.
Now let us consider a $C^1$ function $\psi(t)$ such that
$$\psi (t)\le \underline{w}(t,0)$$
with equality at $t=t_0$. Because of (\ref{eq::g62}) and (\ref{eq::g63}), we see that 
$$\varphi(t,x)=\psi(t) + \bar p^0_L x 1_{\left\{x<0\right\}} + \bar p^0_R x 1_{\left\{x>0\right\}}$$
satisfies
$$\varphi\le \underline{w}$$
with equality at $(t_0,0)$. This implies (\ref{eq::g64}), and at almost every point $t_0$ 
where the Lipschitz continuous function $\underline{w}(t,0)$ is differentiable, we have
$$\underline{w}_t(t_0,0) + a(t_0) \ge \bar A.$$
Because $w$ is $1$-periodic in time, we get after an integration on one period,
\begin{equation}\label{eq::re2}
\langle a \rangle\ge \bar A.
\end{equation}
Together with (\ref{eq::g66}), we deduce that $\langle a \rangle =\bar
A$, which is the desired result, for $N=1$.
\end{proof}

\begin{proof}[Proof of  \eqref{eq:Nge1} in Theorem~\ref{thm:conv-time}].  
We simply remark, using the sub-solution viscosity inequality at each
  junction condition, that for $\alpha=1,\dots,N$, 
\[\bar A \ge \langle a_\alpha \rangle\] 
which is the desired result. This achieves the proof of  \eqref{eq:N1}
and \eqref{eq:Nge1}. 
\end{proof}

\begin{proof}[Proof of the monotonicity of $\bar A$ in Theorem~\ref{thm:conv-time}]

Let $N\ge 2$, and for $i=c,d$, let us assume some given $b_1^i<\dots < b_N^i$.
and let us call $w^i$ a global corrector given by Lemma \ref{lem:cor-ok} (or
Theorem~\ref{thm:corrector}) with $\lambda=\bar A^i$ and $H=H^i$ with
$i=c,d$ respectively. 

\noindent We call $\ell_\alpha^i = b_{\alpha+1}^i -b_\alpha^i>0$ and assume that
$$0<\ell_{\alpha_0}^d - \ell_{\alpha_0}^c=:\delta_{\alpha_0}  \quad \mbox{for some}\quad \alpha_0 \in \left\{1,\dots,N-1\right\}$$
and
$$\ell_{\alpha}^d  = \ell_{\alpha}^c \quad \mbox{for all}\quad \alpha \in \left\{1,\dots,N-1\right\}\backslash \left\{\alpha_0\right\}.$$
Calling $\bar p^0_{\alpha_0}$ a point of global minimum of $\bar H_{\alpha_0}$, we define
$$\tilde{w}^d(t,x)=\left\{\begin{array}{ll}
w^c(t,x-b_{\alpha_0}^d+b_{\alpha_0}^c) & \quad \mbox{if}\quad x\le b_{\alpha_0}^d + \ell_{\alpha_0}^c/2=:x_-,\\
w^c(t,x_--b_{\alpha_0}^d+b_{\alpha_0}^c) +\bar p^0_{\alpha_0}(x-x_-)& \quad \mbox{if}\quad x_- \le x \le x_+,\\
w^c(t,x-b_{\alpha_0+1}^d+b_{\alpha_0+1}^c) +\bar p^0_{\alpha_0}(x_+-x_-)& \quad \mbox{if}\quad x\ge b_{\alpha_0+1}^d-\ell_{\alpha_0}^c/2=:x_+.
\end{array}\right.$$
Recall that $w^i$ for $i=c,d$, are globally Lipschitz continuous in space and
time. This shows that $\tilde{w}^d$ is also Lipschitz continuous in space and
time by construction, because it is continuous at
$x=x_-,x_+$. Moreover $\tilde{w}^d$ is $1$-periodic in time.  We now
want to check that $\tilde{w}^d$ is a sub-solution of the equation
satisfied by $w^d$ with $\bar A^c$ on the right hand side instead of
$\bar A^d$. We only have to check it for all times $\bar t\not\in
\left\{\tau_0,\dots, \tau_K\right\}$ and $\bar x\in
\left[x_-,x_+\right]$, \textit{i.e.} we have to show that
\begin{equation}\label{eq::jr0}
\tilde{w}^d_t(\bar t,\bar x) + \bar H_{\alpha_0}(\tilde{w}^d_x(\bar t,\bar x))\le \bar A^c \quad
\mbox{for all}\quad \bar x\in [x_-,x_+].
\end{equation}
Assume that $\varphi$ is a test function touching $\tilde{w}^d$ from
above at such a point $(\bar t,\bar x)$ with $\bar x\in [x_-,x_+]$.
Then this implies in particular that $\psi(t,x)=\varphi(t,x)-\bar
p_{\alpha_0}^0(x-x_-)$ touches $\tilde{w}^d(\cdot,x_-)=w^c(\cdot,x_0)$
from above at time $\bar t$ with $x_0=b_{\alpha_0}^c +
\ell_{\alpha_0}^c/2$.  Recall that $w^c$ is solution of
$$w^c_t + \bar H_{\alpha_0}(w^c_x)=\bar A^c \quad \mbox{on}\quad (b^c_{\alpha_0},b^c_{\alpha_0+1}).$$
From the characterization of sub-solutions (see Theorem 2.10 in \cite{im}), we  then deduce that
$$\psi_t(\bar t)+\bar H_{\alpha_0}(\bar p^0_{\alpha_0}) \le \bar  A^c.$$
If $\bar x \in (x_-,x_+)$, then we have $\varphi_x(\bar t,\bar x)=\bar p^0_{\alpha^0}$. This means in particular
\begin{equation}\label{eq::jr1}
\varphi_t + \bar H_{\alpha_0}(\varphi_x) \le \bar A^c \quad \mbox{at}\quad (\bar t,\bar x) \quad \mbox{if}\quad \bar x \in (x_-,x_+).
\end{equation}
Using now (\ref{eq::jr1}), and still from Theorem 2.10 in \cite{im}, we deduce that we have in the viscosity sense
\begin{equation}\label{eq::jr2}
\tilde{w}^d_t(\bar t,\bar x) + \max\left(\bar H^-_{\alpha_0}(\tilde{w}^d_x(\bar t,\bar x^+)),\bar H^+_{\alpha_0}(\tilde{w}^d_x(\bar t,\bar x^-))\right)\le \bar A^c \quad
\mbox{for}\quad \bar x = x_\pm.
\end{equation}
Therefore (\ref{eq::jr1}) and (\ref{eq::jr2}) imply  (\ref{eq::jr0}).

Let us now call $H^d$ the Hamiltonian in assumption \textbf{(C1)}
constructed with the points $\{ b_\alpha^d \}_{\alpha=1, \dots, N}$. Then we have
\[\tilde{w}^d_t + H^d(t,x,\tilde{w}^d_x)\le \bar A^c \quad \mbox{for all}\quad t\not\in \left\{\tau_0,\dots \tau_K\right\}.\]
Note that the proof of Theorem \ref{thm:bar-A} is unchanged for the
present problem, and then Theorem \ref{thm:bar-A} still holds true.
This shows that
\begin{equation}\label{eq::jr3}
\bar A^d \le \bar A^c
\end{equation}
which shows the expected monotonicity. The proof is now complete. 
\end{proof}
\begin{remark}
  Note that, in the previous proof, it would also be possible to
  compare the sub-solution given by the restriction of $\tilde{w}^d$
  on some interval $[-\rho,\rho]$ with $\rho>0$ large enough (see
   \cite[Proposition 2.19]{im}), with the approximation $w^{d,\rho}$
  of $w^d$ on $[-\rho,\rho]$ with $\bar A^d\ge \bar A^d_\rho \to \bar
  A^d$ as $\rho\to +\infty$. The comparison for large times would imply
  $\bar A^d_\rho \le \bar A^c$.  As $\rho\to +\infty$, this would give
  the same conclusion (\ref{eq::jr3}).
\end{remark}

\begin{proof}[Proof of \eqref{eq:critical} in Theorem~\ref{thm:conv-time}]
Let $w$ be a global corrector associated to  $\bar A$.

Recall that
\begin{equation}\label{eq::jr10}
\bar A \ge \bar A_0:= \max_{\alpha=1,\dots,N} \langle a_\alpha\rangle
\ge  A_0:=\max_{\alpha=1,\dots,N} A_0^\alpha 
\quad \mbox{with}\quad A^\alpha_0=\max_{\beta=\alpha-1,\alpha} (\min \bar H_\beta).
\end{equation}
Our goal is to prove that $\bar A=\bar A_0$ when all the distances
$\ell_\alpha$ are large enough, \textit{i.e.} (\ref{eq:critical}).
Let us assume that
$$\bar A > \bar A_0.$$

\noindent \textsc{Step 1: Considering another corrector with the same $\langle \widehat a_\alpha \rangle = \bar A_0$.} 
Let $\mu_\alpha\ge 0$ such that
$$\widehat a_\alpha = \mu_\alpha + a_\alpha \quad \mbox{with}\quad \langle \widehat a_\alpha\rangle = \bar A_0 
\quad \mbox{for all}\quad \alpha=1,\dots,N.$$
Let us call $\widehat w$ the corresponding corrector with associated constant $\widehat A$.
Then Theorem \ref{thm:bar-A} (still valid here) implies that 
$$\widehat A\ge \bar A > \bar A_0.$$
We also split the set $\left\{1,\dots,N\right\}$ into two disjoint sets
$$I_0=\left\{\alpha\in \left\{1,\dots,N\right\},\quad \bar A_0 = A^\alpha_0\right\}$$
and
$$I_1=\left\{\alpha\in \left\{1,\dots,N\right\},\quad \bar A_0 > A^\alpha_0\right\}.$$
Note that by (\ref{eq::jr10}), if $\alpha\in I_0$, then $\langle a_\alpha\rangle = A^\alpha_0$, and then by {\bf (C3)}, we have
$a_\alpha(t)=const=A^\alpha_0$ for all time $t\in \R$. For later use, we then claim that $\widehat w$ satisfies
\begin{equation}\label{eq::jr11}
\widehat w_t(t,x)+ \max(\bar H_\alpha^-(\widehat w_x(t,x^+)), \bar
H_{\alpha-1}^+(\widehat w_x(t,x^-)))= \hat A \quad \mbox{for all}\quad (t,x)\in \R\times \left\{b_\alpha\right\}
\end{equation}
and not only for $t\in \R\backslash \left(\Z +
  \left\{\tau_0,\dots,\tau_K\right\}\right)$.  Let us show it for
sub-solutions (the proof being similar for super-solutions).  Let
$\varphi$ be a test function touching $\widehat w$ from above at some
point $(\bar t,\bar x)=(j+\tau_k,b_\alpha)$ for some $j\in\Z$, $k\in
\left\{0,\dots,K\right\}$. Assume also that the contact between
$\varphi$ and $\widehat w$ only holds at that point $(\bar t,\bar x)$.
The proof is a variant of a standard argument. For $\eta>0$, let us
consider the test function
$$\varphi_\eta(t,x)=\varphi(t,x)+\frac{\eta}{\bar t-t} \quad \mbox{for}\quad t\in (-\infty,\bar t).$$
Then  for $r>0$ fixed, we have
$$\inf_{(t,x)\in \overline{B_r(\bar t,\bar x)},\ t<\bar t} (\varphi_\eta-\widehat w)(t,x)=(\varphi_\eta-\widehat w)(t_\eta,x_\eta)$$
with
\[\left\{\begin{array}{l}
P_\eta=(t_\eta,x_\eta)\to (\bar t,\bar x)=\bar P \quad \text{ as } \quad
\eta\to 0, \medskip \\
\varphi_t(\bar P) \le \limsup_{\eta \to 0}(\varphi_\eta)_t(P_\eta).
\end{array}\right.\]
This implies that $\widehat w$ is a relaxed viscosity sub-solution at
$(\bar t,\bar x)$ in the sense of Definition~2.2 in \cite{im}.  By
 \cite[Proposition 2.6]{im}, we deduce that $\widehat w$ is also a
standard (\textit{i.e.} not relaxed) viscosity sub-solution at $(\bar
t,\bar x)$.  Finally we get (\ref{eq::jr11}). \medskip

\noindent \textsc{Step 2: Defining a space super-solution.} 
Let us define the function
\[M(x)=\inf_{t\in \R} \widehat w(t,x).\]
Because $\widehat w$ is globally Lispschitz continuous, we deduce that $M$ is also globally Lipschitz continuous.
Moreover we have the following viscosity super-solution  inequality
\[\bar H_{\alpha}(M_x(x))\ge \widehat A > \bar A_0\quad \mbox{for all}\quad x\in (b_{\alpha},b_{\alpha+1}),
\quad \mbox{for all}\quad \alpha=0,\dots,N.\]
Let us call for $\alpha=0,\dots,N$:
\[\bar p_{\alpha,R} = \min E_{\alpha,R} \quad \mbox{with}\quad E_{\alpha,R}=\left\{p\in\R,\quad \bar H^+_{\alpha}(p)=\bar H_{\alpha}(p)=\bar A_0\right\},\]
\[\bar p_{\alpha,L} = \max E_{\alpha,L} \quad \mbox{with}\quad E_{\alpha,L}=\left\{p\in\R,\quad \bar H^-_{\alpha}(p)=\bar H_{\alpha}(p)=\bar A_0\right\}.\]
Let us now consider $\alpha =0,\dots, N$ and two points $x_-<x_+$ with $x_\pm \in (b_\alpha,b_{\alpha+1})$.
Let us assume that there is a test function $\varphi^\pm$ touching $M$ from below at $x_\pm$. Then we have
$$\bar H_\alpha(\varphi^\pm_x(x_\pm))\ge \widehat A>\bar A_0$$
with 
$$\varphi^\pm_x(x_\pm)\ge \bar p_{\alpha,R}  \quad \mbox{or}\quad \varphi^\pm_x(x_\pm)\le \bar p_{\alpha,L}.$$
Moreover, if $\bar A_0> \min \bar H_\alpha$, then we have
$$\bar p_{\alpha,L}<\bar p^0_\alpha < \bar p_{\alpha,R}$$
for any $\bar p^0_{\alpha}$ which is a point of global minimum of $\bar H_{\alpha}$.\medskip

\noindent \textsc{Step 3: A property of the space super-solution.}  We
now claim that the following case is impossible:
$$p^-:=\varphi^-_x(x_-)< \varphi^+_x(x_+)=:p^+ \quad \mbox{and}\quad \inf_{[p^-,p^+]}\bar H_\alpha < \widehat A.$$
If it is the case, then let $\bar p\in (p^-,p^+)$ such that $\bar
H_\alpha(\bar p)<\widehat A$.  Therefore the geometry of the graph of the
function $M$ implies that
$$\inf_{x\in [x_-,x_+]} (M(x)- x \bar p) = M(\bar x)- \bar x \bar
p\quad \mbox{for some}\quad \bar x\in (x_-,x_+)$$
and then we have the viscosity super-solution inequality at $\bar x$:
$$\bar H_\alpha(\bar p)\ge \widehat A$$
which leads to a contradiction.  Therefore (in all cases $\bar A_0
>\min \bar H_\alpha$ or $\bar A_0=\min \bar H_\alpha$), it is possible
to check that there is a point $\bar x_\alpha\in
[b_{\alpha},b_{\alpha+1}]$ such that the Lipschitz continuous function $M$
satisfies in the viscosity sense
$$\left\{\begin{array}{ll}
M_x\ge \bar p_{\alpha,R} &\quad \mbox{in}\quad (b_\alpha,\bar
x_\alpha),\\ - M_x \ge - \bar p_{\alpha,L} &\quad \mbox{in}\quad (\bar
x_\alpha,b_{\alpha+1}).
\end{array}\right.$$ Moreover from Theorem \ref{thm:corrector} ii)
(see Lemma \ref{lem:cor-ok}), we deduce from $\widehat A>\max (\min
H_N, \min H_0)$ that
$$\bar x_{N}=+\infty \quad \mbox{and}\quad \bar x_0=-\infty.$$
In particular, we deduce that there exists at least one $\alpha_0\in
\left\{1,\dots,N\right\}$ such that
\begin{equation}\label{eq::jr5} \bar x_{\alpha_0} - b_{\alpha_0}\ge
\ell_{\alpha_0}/2\quad \mbox{and}\quad b_{\alpha_0} - \bar
x_{\alpha_0-1}\ge \ell_{\alpha_0-1}/2.
\end{equation}

\noindent \textsc{Step 4: The case $\alpha_0\in I_0$.} 
In this case, we see that there exists a time $\bar t$ such that the test function
$$\varphi(t,x)=\left\{\begin{array}{ll}
\bar p_{\alpha_0,R}(x-b_{\alpha_0}) & \quad \mbox{for}\quad x\ge b_{\alpha_0},\\
\bar p_{\alpha_0-1,L}(x-b_{\alpha_0}) & \quad \mbox{for}\quad x\le b_{\alpha_0}
\end{array}\right.$$
is a test function touching (up to some additive constant) $\widehat w$  from below at $(\bar t,b_{\alpha_0})$. 
By (\ref{eq::jr11}), this implies
$$\bar A_0=\max(\bar H_{\alpha_0}(\bar p_{\alpha_0,R}),\bar H_{\alpha_0-1}(\bar p_{\alpha_0-1,L})) \ge \widehat A\ge \bar A.$$
Contradiction. \medskip

\noindent \textsc{Step 5: Consequences on $\widehat w$.} 
From the fact that $\widehat w$ is $1$-periodic in time and $C$-Lipschitz continuous in time (with a constant $C$ 
depending only on $\displaystyle \max_{\alpha=1,\dots,N}\|\widehat a_\alpha\|_\infty$ and the $\bar H_\alpha$'s, see (\ref{eq::g30})),
we deduce that we have
\begin{equation}\label{eq::jr6}
\left\{\begin{array}{ll}
\widehat w(t,x+h)-\widehat w(t,x)\ge \bar p_{\alpha,R} h -2C &\quad \mbox{for}\quad x,x+h \in (b_\alpha,\bar x_\alpha),\\
\widehat w(t,x-h)-\widehat w(t,x)\ge -\bar p_{\alpha,R} h -2C &\quad \mbox{for}\quad x,x+h \in (\bar x_\alpha,b_{\alpha+1}).
\end{array}\right.
\end{equation}

\noindent \textsc{Step 6: The case $\alpha_0\in I_1$: definition of a space-time  super-solution.} 
Proceeding similarly to Step 3 of the proof of \eqref{eq:N1},
we define
$$\underline{\widehat w}_{\alpha_0,R}(t,x)=\inf_{\frac{\ell_{\alpha_0}}{4}\ge h\ge 0} \left(\widehat w(t,x+h)-\bar p^0_{\alpha_0}h\right)\quad \mbox{for}\quad b_{\alpha_0}\le x \le b_{\alpha_0} +\frac{\ell_{\alpha_0}}{4}$$
and
$$\underline{\widehat w}_{(\alpha_0-1),L}(t,x)=\inf_{\frac{\ell_{(\alpha_0-1)}}{4}\ge  h\ge 0} \left(\widehat w(t,x-h)+\bar p^0_{\alpha_0-1}h\right)\quad \mbox{for}\quad    b_{\alpha_0} - \frac{\ell_{\alpha_0-1}}{4} \le x\le b_{\alpha_0}.$$
From (\ref{eq::jr6}), we deduce that we have for some $\bar h\in [0,\frac{\ell_{\alpha_0}}{4}]$
$$\widehat w(t,x)\ge \underline{\widehat w}_{\alpha_0,R}(t,x) = \widehat w(t,x+\bar h)-\bar p^0_{\alpha_0}\bar h
\ge \widehat w (t,x)+(\bar p_{\alpha_0,R}-\bar p^0_{\alpha_0}) \bar h -2C$$
which implies
$$0\le  \bar h\le \frac{2C}{\bar p_{\alpha_0,R}-\bar p^0_{\alpha_0}}.$$
As in Step 3 of the proof of \eqref{eq:N1}, if
\begin{equation}\label{eq::jr7}
\frac{\ell_{\alpha_0}}{4} >  \frac{2C}{\bar p_{\alpha_0,R}-\bar p^0_{\alpha_0}}
\end{equation}
this implies that $\underline{\widehat w}_{\alpha_0,R}$ is a super-solution for $x\in (b_{\alpha_0},b_{\alpha_0}+\frac{\ell_{\alpha_0}}{4})$.
Similarly, if 
\begin{equation}\label{eq::jr8}
\frac{\ell_{\alpha_0-1}}{4} >  \frac{2C}{\bar p^0_{\alpha_0-1}-\bar p_{\alpha_0-1,L}}
\end{equation}
then $\underline{\widehat w}_{\alpha_0-1,L}$ is a super-solution for $x\in (b_{\alpha_0}-\frac{\ell_{\alpha_0-1}}{4},b_{\alpha_0})$.
We now define
$$\underline{\widehat w}(t,x)=\left\{\begin{array}{ll}
\underline{\widehat w}_{\alpha_0,R}(t,x) & \quad \mbox{if}\quad x\in (b_{\alpha_0},b_{\alpha_0}+\frac{\ell_{\alpha_0}}{4}),\\
\underline{\widehat w}_{\alpha_0-1,L}(t,x) & \quad \mbox{if}\quad x\in (b_{\alpha_0}-\frac{\ell_{\alpha_0-1}}{4},b_{\alpha_0}),\\
\min (\underline{\widehat w}_{\alpha_0-1,L}(t,b_{\alpha_0}), \ \underline{\widehat w}_{\alpha_0,R}(t,b_{\alpha_0})) & \quad \mbox{if}\quad x= b_{\alpha_0}.
\end{array}\right.$$
Then as in Steps 4 and 5 of the proof \eqref{eq:N1}, we deduce that $\underline{\widehat w}$ is a super-solution up to the junction point $x=b_{\alpha_0}$ and that
$$\bar A_0= \langle \widehat a_{\alpha_0}\rangle \ge \widehat A \ge \bar A.$$
Contradiction. \medskip

\noindent \textsc{Step 7: Conclusion.} 
If (\ref{eq::jr7}) and (\ref{eq::jr8}) hold true for any $\alpha_0\in I_1$, then we deduce that $\bar A \le \bar A_0$, which implies $\bar A = \bar A_0$.
This ends the proof of \eqref{eq:critical} in Theorem~\ref{thm:conv-time}.
\end{proof}

\begin{proof}[Proof of \eqref{eq:limit} in Theorem~\ref{thm:conv-time}]
Let us consider 
$$\bar a(t)=\max_{\alpha=1,\dots,N} a_\alpha(t),$$ 
and $(w,\bar{\bar{A}})$ a solution (given by Theorem \ref{thm:corrector} (see also Lemma \ref{lem:cor-ok})) of
$$\left\{\begin{array}{ll}
w_t + \bar H_0(w_x) = \bar{\bar{A}}& \quad \mbox{if}\quad x<0,\\
w_t + \bar H_N(w_x) = \bar{\bar{A}}& \quad \mbox{if}\quad x>0,\\
w_t(t,0)+\max (\bar a(t),\bar H_N^-(w_x(t,0^+)), \bar
H_0^+(w_x(t,0^-))) = \bar{\bar{A}} & \quad \mbox{if}\quad x=0, \medskip\\
w \quad \mbox{is $1$-periodic with respect to $t$}.
\end{array}\right.$$
From Theorem~\ref{thm:conv-time}, we also know that
$$\bar{\bar{A}} = \langle \bar a \rangle.$$
For $N\ge 2$, we set $\ell=(\ell_1,\dots,\ell_{N-1})\in (0,+\infty)^{N-1}$ and consider $b_0=-\infty<b_1<\dots<b_N<b_{N+1}=+\infty$
with
$$\ell_\alpha=b_{\alpha+1}-b_\alpha \quad \mbox{for}\quad \alpha=1,\dots N-1.$$
We now call $(w^\ell,\bar A^\ell)$ a global corrector given by Theorem \ref{thm:corrector} (see also Lemma \ref{lem:cor-ok}).
The remaining of the proof is divided into several steps. \medskip

\noindent \textsc{Step 1: Bound from above on $\bar A^\ell$.}
We define
$$\tilde{w}(t,x)=\left\{\begin{array}{ll}
w(t,x-b_{1}) & \quad \mbox{if}\quad x\le b_{1},\\
\\
\displaystyle w(t,0)+ \bar p^0_\alpha(x-b_\alpha)  +
\sum_{\beta=1,\dots,\alpha-1} \bar p^0_\beta (b_{\beta+1}-b_\beta) & 
\quad \mbox{if}\quad \left\{\begin{array}{l}
b_{\alpha} \le x\le b_{\alpha+1},\\
\alpha\in \left\{1,\dots,N-1\right\},\\
\end{array}\right.\\
\\
\displaystyle w(t,x-b_N)  + \sum_{\beta=1,\dots,N-1} \bar p^0_\beta (b_{\beta+1}-b_\beta) & \quad \mbox{if}\quad x\ge b_N.
\end{array}\right.$$
Proceeding  as in Step 1 of the proof of Theorem~\ref{thm:conv-time} ii),
it is then easy to check that $\tilde{w}$ is a sub-solution of the
equation satisfied by 
$w^\ell$ with $\bar{\bar{A}}$ on the right hand side instead of $\bar A^\ell$. Then Theorem \ref{thm:bar-A} implies that
\begin{equation}\label{eq::jr15}
\bar A^\ell \le \bar{\bar{A}} = \langle \bar a \rangle.
\end{equation}

\noindent \textsc{Step 2: Bound from below on $\bar A^\ell$.}
From Theorem 2.10 in \cite{im}, we deduce that we have in the viscosity sense (in time only)
$$w^\ell_t(t,b_\alpha)+ a_\alpha(t)\le \bar A^\ell \quad \mbox{for
  all}\quad t \notin  \cup_{k=0}^K \{ \tau_k + \Z\}.$$
Let us call
$$\underline{A}=\liminf_{\ell \to 0} \bar A^\ell.$$
We also know that $w^\ell$ is $1$-periodic and  globally Lipschitz continuous with a constant which is independent on $\ell$.
Therefore there exists a $1$-periodic and Lipschitz continuous function $g=g(t)$ such that
$$w^\ell(t,b_\alpha) \to g(t) \quad \mbox{for all}\quad \alpha=1,\dots,N,\quad \mbox{as}\quad \ell\to 0.$$
The stability of viscosity solutions implies in the viscosity sense
$$g'(t) + a_\alpha(t) \le \underline{A},\quad \mbox{for all}\quad
\alpha=1,\dots,N,\quad \mbox{for all} \quad t \notin  \cup_{k=0}^K \{ \tau_k + \Z\}.$$
Because $g$ is Lipschitz continuous, this inequality also holds for almost every $t\in\R$. This implies
$$g'(t) + \bar a(t)\le  \underline{A} \quad \mbox{for a.e.}\quad t\in\R.$$
An integration on one period gives
\begin{equation}\label{eq::jr16}
\langle \bar a \rangle \le \underline{A}.
\end{equation}

\noindent \textsc{Step 3: Conclusion.}  Combining (\ref{eq::jr15})
with (\ref{eq::jr16}) finally yields that $\bar A^\ell \to \langle \bar
a \rangle$ as $\ell\to 0$.  The proof of \eqref{eq:limit} in
Theorem~\ref{thm:conv-time} is now complete.
\end{proof}

\paragraph{Acknowledgements.}
The authors thank the referees for their valuable comments.  The
authors thank Y. Achdou, K. Han and N. Tchou for stimulating
discussions. The authors thank N. Seguin for interesting discussions
on green waves. C.~I.  thanks Giga for the interesting discussions
they had together and for drawing his attention towards papers such as
\cite{hamamuki}.  R.~M. thanks G. Costeseque for his comments on
traffic lights modeling and his simulations which inspired certain
complementary results. The second and third authors are partially
supported by ANR-12-BS01-0008-01 HJnet project.


\appendix 

\section{Proofs of some technical results}

\subsection{The case $\bar x \neq 0$ in the proof of convergence}

\begin{proof}[The case $\bar x \neq 0$ in the proof of
    Theorem~\ref{thm:conv}]

    We only deal with the subcase $\bar x >0$ since the subcase $\bar x<0$ is
    treated in the same way.  Reducing $\overline r$ if necessary, we
    may assume that $B_{\overline r}(\overline t,\overline x)$ is
    compactly embedded in the set
    $\left\{(t,x)\in(0,+\infty)\times(0,+\infty):\,x>0\right\}$: there
    exists a positive constant $c_{\overline r}$ such that
\begin{equation}\label{conv1}
(t,x)\in B_{\overline r}(\overline t,\overline x)\quad\Rightarrow \quad x>c_{\overline r}\,.
\end{equation}
Let $p=\varphi_x(\overline t,\overline x)$ and let $v^R=v^R(t,x)$  be a solution of the cell problem 
\begin{equation}\label{conv3}
v_t^R+H_R\left(t,x,p+v^R_x\right)=\bar H_R(p)\quad\text{in}\quad\R\times\R\,.
\end{equation}
We claim that if $\varepsilon>0$ is small enough, the perturbed test function \cite{evans}
$$\varphi^\varepsilon(t,x)=\varphi(t,x)+\varepsilon v^R\left(\frac t\varepsilon,\frac x\varepsilon\right)$$
satisfies, in the viscosity sense, the inequality
\begin{equation}\label{conv4}
\varphi^\varepsilon_t+H\left(\frac t\varepsilon,\frac x\varepsilon,\varphi_x^\varepsilon\right)
\geq\frac\theta2\quad\text{in}\quad B_r(\overline t,\overline x)
\end{equation}
for sufficiently small $r>0$. To see this let $\psi$ be a test
function touching $\varphi^\varepsilon$ from below at $(t_1,x_1)\in
B_r(\overline t,\overline x)\subseteq B_{\overline r}(\overline
t,\overline x)$. In this way the function
\[\eta(s,y)=\frac{1}{\varepsilon}\left(\psi(\varepsilon s,\varepsilon y)
-\varphi(\varepsilon s, \varepsilon y)\right)\]
touches $v^R$ from below at
$(s_1,y_1)=\left(\frac{t_1}{\varepsilon},\frac{x_1}{\varepsilon}\right)$
and (\ref{conv3}) yields
\begin{equation}\label{conv5}
\psi_t(t_1,x_1)-\varphi_t(t_1,x_1)+H_R\left(\frac{t_1}{\varepsilon},\frac{x_1}{\varepsilon},p
+\psi_x(t_1,x_1)-\varphi_x(t_1,x_1)\right)\geq \bar H_R(p).
\end{equation}
Since (\ref{conv1}) implies that $\frac x\varepsilon\to+\infty$, as
$\varepsilon\to0$, uniformly with respect to $(t,x)\in B_{\overline
  r}(\overline t,\overline x)$, we can find, owing to \textbf{(A5)},
an $\varepsilon_0>0$ independent of $\psi$ and $(t_1,x_1)$ such that
the inequality
\begin{equation}\label{conv6}
H\left(\frac{t_1}{\varepsilon},\frac{x_1}{\varepsilon},\psi_x(t_1,x_1)\right)\geq
H_R\left(\frac{t_1}{\varepsilon},\frac{x_1}{\varepsilon},\psi_x(t_1,x_1)\right)
-\frac\theta4
\end{equation}
holds true for $\varepsilon<\varepsilon_0$. Combining
(\ref{conv2})-(\ref{conv5})-(\ref{conv6}) and using the continuity of
$\varphi_x$ and $\varphi_t$ we have
\begin{equation*}
\begin{split}
&\psi_t(t_1,x_1)+H\left(\frac{t_1}{\varepsilon},\frac{x_1}{\varepsilon},\psi_x(t_1,x_1)\right)\\
&\geq\psi_t(t_1,x_1)+H_R\left(\frac{t_1}{\varepsilon},\frac{x_1}{\varepsilon},p+\psi_x(t_1,x_1)-\varphi_x(t_1,x_1)\right)+\\
&\quad
  H_R\left(\frac{t_1}{\varepsilon},\frac{x_1}{\varepsilon},\psi_x(t_1,x_1)\right)-H_R\left(\frac{t_1}{\varepsilon},\frac{x_1}{\varepsilon},\varphi_x(\bar
  t, \bar x)+\psi_x(t_1,x_1)-\varphi_x(t_1,x_1)\right)-\frac\theta4\\
&\geq\frac\theta2
\end{split}
\end{equation*}
if $r$ is sufficiently close to 0. The claim (\ref{conv4}) is
proved. 

Since $\varphi$ is strictly above $\overline u$, if $\varepsilon$ and $r$ are small enough
\[u^\varepsilon+\kappa_r\leq\varphi^\varepsilon\quad\text{on}\quad\partial B_r(\overline t, \overline x)\]
for a suitable positive constant $\kappa_r$. By comparison principle we deduce
\[u^\varepsilon+\kappa_r\leq\varphi^\varepsilon\quad\text{in}\quad B_r(\overline t, \overline x)\]
and passing to the limit as $\eps \to 0$ and $(t,x) \to (\bar t,\bar
x)$ on both sides of the previous inequality, we produce the
contradiction
\[\overline u(\overline t,\overline x)<\overline u(\overline t,\overline x)
+\kappa_r\leq \varphi(\overline t,\overline x)=\overline u(\overline
t,\overline x)\,. \qedhere\]
\end{proof}

\subsection{Proof of Lemma~\ref{lem:unique}}

\begin{proof}[Proof of Lemma~\ref{lem:unique}]
We first adress uniqueness. Let us assume that we have two solutions
$(v^i,\lambda^i)$ for $i=1,2$ of (\ref{eq:cell-alpha}). Let
$$u^i(t,x)=v^i(t,x)+px - \lambda^i t$$
Then $u^i$ solves
$$u^i_t + H_\alpha(t,x,u^i_x)=0$$
with
$$u^1(0,x)\le u^2(0,x) +C$$
The comparison principle implies
$$u^1\le u^2 +C \quad \mbox{for all}\quad t>0$$
and then $\lambda^1\ge \lambda^2$. Similarly we get the reverse inequality and then $\lambda^1=\lambda^2$.

We now turn to the continuity of the map $p\mapsto \bar
H_\alpha(p)$. It follows from the stability of viscosity sub- and
super-solutions, from the fact that the constant $C$ in \eqref{eq::g4}
is bounded for bounded $p$'s and from the comparison principle.
This achieves the proof of the lemma. \end{proof}

\subsection{Sketch of the proof of Proposition~\ref{prop:comp}}
\label{sec:comp}

\begin{proof}[Sketch of the proof of Proposition~\ref{prop:comp}]
Consider 
\[ M_\nu = \sup_{x \in [\rho_1,\rho_2], s,t \in \R} \left\{
u(t,x) - v(s,x) - \frac{(t-s)^2}{2\nu} \right\}. \]
We want to prove that 
\[ M = \lim_{\nu \to 0} M_\nu \le 0. \]
We argue by contradiction by assuming that $M>0$. The supremum
defining $M_\nu$ is reached, let  $s_\nu,t_\nu,x_\nu$ denote a maximizer. 
Choose $\nu$ small enough so that $M_\nu \ge \frac{M}2 >0$. We classically get,
\[ |t_\nu - s_\nu| \le C \sqrt{\nu}.\]

If there exists $\nu_n \to 0$ such that $x_{\nu_n} = \rho_1$ for all
$n \in \N$, then \[ \frac{M}2 \le M_{\nu_n} \le
  U_0(t_{\nu_n}) -U_0(s_{\nu_n}) \le \omega_0 (t_{\nu_n} -s_{\nu_n})
  \le \omega_0 (C \sqrt{\nu_n})\] where $\omega_0$ denotes the
modulus of continuity of $U_0$. The contradiction $M\le 0$ is obtained
by letting $n$ go to $+\infty$.

Hence, we can assume that for $\nu$ small enough, $x_\nu > \rho_1$.
Reasoning as in \cite[Theorem~7.8]{im}, we can easily reduce to the
case where $H(t_\nu,x_\nu,\cdot)$ reaches its minimum for
$p=p_0=0$. We can also consider the vertex test function $G^\gamma$
associated with the single Hamiltonian $H$ (using notation of
\cite{im}, it corresponds to the case $N=1$) and the free parameter
$\gamma$.  If $x_\nu<\rho_2$, then $G^\gamma(x,y)$ reduces to the
standard test function $\frac{(x-y)^2}{2}$.

We next consider 
\[ M_{\nu,\eps} = \sup_{\stackrel{x,y \in [\rho_1,\rho_2] \cap \overline{B_r(x_\nu)}}{s,t \in \R}} \left\{
u(t,x) - v(s,y) - \frac{(t-s)^2}{2\nu} - \eps G^\gamma (\eps^{-1}x,\eps^{-1}y)
- \varphi^\nu (t,s,x) \right\} \]
where $r=r_\nu$ is chosen so that $\rho_1 \notin \overline{B_r (x_\nu)}$ and 
 the localization function
\[ \varphi^\nu (t,s,x) = \frac12((t-t_\nu)^2 + (s-s_\nu)^2 +
(x-x_\nu)^2).\]
The supremum defining $M_{\nu,\eps}$ is reached and if $(t,s,x,y)$
denotes a maximizer, then 
\[ (t,s,x,y) \to (t_\nu,s_\nu,x_\nu,x_\nu) \quad \text{ as }
(\eps,\gamma) \to 0 .\]
In particular, $x,y \in B_r (x_\nu)$ for $\eps$ and $\gamma$ small
enough. The remaining of the proof is completely analogous (in fact
much simpler). 
\end{proof}

\subsection{Construction of $\lambda_\rho$ in the proof of Lemma~\ref{lem:cor-ok}}

In order to get $\lambda_\rho$, it is enough to apply the following
lemma. 
\begin{lemma}
Let $u$ be the solution of a Hamilton-Jacobi equation of evolution-type
submitted to the initial condition: $u (0,x) = 0$ and posed
on a compact set $K$. Assume that 
\begin{itemize}
\item the comparison principle holds true;
\item $u$ is $L$-globally Lipschitz continuous in time and space;
\item $u(k+\cdot,\cdot)+C$ is a solution for all $k \in \N$ and $C \in
  \R$.
\end{itemize}
There then exists $\lambda \in \R$ such that 
\[ |u(t,x) - \lambda t | \le C_0  \]
and 
\[ |\lambda | \le L \]
where  $C_0 =  L (2+3\rho)$ if $\rho$ denotes the
diameter of $K$. 
\end{lemma}
\begin{proof}
Define 
\[ \lambda^+ (T) = \sup_{\tau \ge 0 } \frac{u (\tau+T,0) -
  u(\tau,0)}{T} \quad \text{ and } \quad
\lambda^- (T) = \inf_{\tau \ge 0 } \frac{u (\tau+T,0) -
  u(\tau,0)}{T}. \]
Remark that $T \mapsto \pm T \lambda^\pm (T)$ is sub-additive. Remark that
the fact that $u$ is $L$-Lipschitz continuous with respect to time
implies that $\lambda^\pm(T)$ are both finite:
\[  |\lambda^\pm(T)| \le L. \]
 the ergodic theorem
implies that $\lambda^\pm (T)$ converges towards $\lambda^\pm$ and 
\[\lambda^+ = \inf_{T>0} \lambda^+ (T) \quad \text{ and } \quad 
\lambda^- = \sup_{T>0} \lambda^-(T).\] If moreover
\begin{equation}\label{estim:lambda-pm}
 | \lambda^+ (T) -  \lambda^-(T) | \le \frac{C}T,
\end{equation}
then the proof of the lemma is complete. Indeed,
\eqref{estim:lambda-pm} implies in particular that
$\lambda^+=\lambda^-$ and 
\[ - \frac{C}T \le \lambda^- (T) - \lambda \le \lambda^+ (T) - \lambda
\le \frac{C}T.\] 
This implies that 
\[ |u(t,0) - \lambda t | \le C.\]
Finally, we get 
\[ |u (t,x) - \lambda t | \le C + L \rho.\] 

It remains to prove \eqref{estim:lambda-pm}. There exists $k\in \Z$
and $\beta \in [0,1)$ such that $\tau^+ = k + \tau^- +
\beta$. Moreover, 
\[ u(\tau^+,x)  \le u (\tau^-+\beta,x) + u(\tau^+,0) -
u(\tau^-+\beta,0) + 2 L \rho \]
where $\rho = \mathrm{diam} \; K$. Now remark that
$u(\tau^-+\beta+t,x)+D$ is a solution in $[\tau^+,+\infty)$ for all
constant $D$. Hence, we get by comparison that for all $t>0$ and $x
\in K$,
\[ u (\tau^++t,x) \le u(\tau^-+\beta +t,x) + u(\tau^+,0) -
u(\tau^-+\beta,0) + 2 L \rho .\]
In particular, 
\begin{align*} 
u (\tau^+ +T,0) -u(\tau^+,0) & \le u(\tau^-+ \beta + T,0) -
u (\tau^-+\beta, 0)  + 2 L \rho \\
& \le u(\tau^-+ T,0) - u (\tau^-, 0) + 2L (1+\rho).
\end{align*} 
Finally, we get (after letting $\eps \to 0$),
\[ \lambda^+(T) \le \lambda^- (T) + \frac{2L(1+\rho)}{T}.\]
Similarly, we can get 
\[ \lambda^+(T) \ge \lambda^- (T) - \frac{2L(1+\rho)}{T}.\]
This implies \eqref{estim:lambda-pm} with $C = 2 L (1+\rho)$.
The proof of the lemma is now complete. 
\end{proof}


\nocite{*}
\bibliographystyle{plain}
\bibliography{gim}

\end{document}